\documentclass[12pt]{article}
\usepackage[amsmath]{e-jc}

\usepackage{graphicx}
\usepackage{tikz}

\newtheorem{sublemma}{}[theorem]

\dateline{Mar 15, 2022}{Oct 21, 2022}{TBD}


\Copyright{The authors. Released under the CC BY-ND license (International 4.0).}

\title{Ordering circuits of matroids}

\author{
Cameron Crenshaw \qquad  James Oxley\\
\small Department of Mathematics\\[-0.8ex]
\small Louisiana State University\\[-0.8ex]
\small Baton Rouge, U.S.A\\
\small\tt ccrens5@lsu.edu \qquad oxley@math.lsu.edu}

\newcommand{\del}{\backslash}
\DeclareMathOperator{\cl}{cl}
\DeclareMathOperator{\si}{si}
\DeclareMathOperator{\co}{co}

\begin{document}

\maketitle


\begin{abstract}
  The cycles of a graph give a natural cyclic ordering to their edge-sets, and these orderings are consistent in that two edges are adjacent in one cycle if and only if they are adjacent in every cycle in which they appear together. An orderable matroid is one whose set of circuits admits such a consistent ordering. In this paper, we consider the question of determining which matroids are orderable. Although we are able to answer this question for non-binary matroids, it remains open for binary matroids. We give examples to provide insight into the potential difficulty of this question in general. We also show that, by requiring that the ordering preserves the three arcs in every theta-graph restriction of a binary matroid $M$, we guarantee that $M$ is orderable if and only if $M$ is graphic.
\end{abstract}


\section{Introduction}
\label{introduction}

In a graph, the edges of each cycle have an ordering on them. But this is not true for the circuits of a matroid. The goal of this paper is to see to what extent we can distinguish graphic matroids by an ordering condition that mimics the ordering condition on the edges of the cycles of a graph.

A \emph{reversible cyclic ordering} of a finite set $X$ is an arrangement of the elements of $X$ on the vertices of an $n$-gon with one element at each vertex. Elements $x_1$ and $x_2$ of $X$ are \emph{adjacent} in the ordering when the corresponding vertices of the $n$-gon lie on a common edge. Figure~\ref{ordnotationex} shows an example of such an ordering $(x_1\ x_2\ \dots\ x_n)$. The same ordering can also be denoted, for example, by $(x_3\ x_2\ x_1\ x_n\ \dots\ x_4)$. Throughout this paper, all orderings are assumed to be reversible cyclic orderings unless stated otherwise.

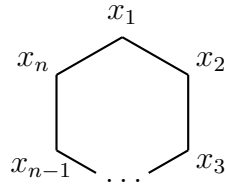
\begin{figure}
\centering
\begin{tikzpicture}[style=thick]

\foreach \x in {0,1,4,5} \draw (90+60*\x:1)--(150+60*\x:1);
\draw[shorten >= 0.4cm] (-30:1)--(-90:1);
\draw[shorten >= 0.4cm] (210:1)--(-90:1);

\draw (90:1.3) node {$x_1$};
\draw (30:1.35) node {$x_2$};
\draw (-30:1.35) node {$x_3$};
\draw (0.05,-0.9) node {$\dots$};
\draw (215:1.3) node {$x_{n-1}$};
\draw (150:1.35) node {$x_n$};

\end{tikzpicture}
\caption{A reversible cyclic ordering.}
\label{ordnotationex}
\end{figure}\noindent

In a graph, there is an associated ordering on the edge set of each cycle. These orderings have the property that two edges are adjacent in an ordering of a given cycle if and only if they are adjacent in the ordering of every cycle in which the edges appear together.

Unlike the cycles of a graph, the circuits of a matroid are sets without inherent order. We give a matroid $M$ an ordering by imposing an ordering on each of its circuits. Such an ordering of $M$ is \emph{consistent} if, for every pair $\{e,f\}$ of distinct elements of $E(M)$ and every pair $\{C,C'\}$ of circuits of $M$ with $\{e,f\}\subseteq C\cap C'$, if $e$ and $f$ are adjacent in the ordering of $C$, then $e$ and $f$ are adjacent in the ordering of $C'$. A matroid is called \emph{orderable} if it has a consistent ordering.

The notation for matroids in this paper follows~\cite{oxley} with one modification. We call a matroid $N$ a~\emph{series extension} of a matroid $M$ if $N$ can be obtained from $M$ by a (possibly empty) sequence of single-element series extensions; a \emph{parallel extension} is defined analogously.

The primary goal of this work is characterizing orderable matroids. As noted above, our first examples of orderable matroids are graphic matroids.

\begin{proposition}
\label{graphicimpord}
If $M$ is a graphic matroid, then $M$ is orderable.
\end{proposition}

However, orderability is not enough to distinguish graphic matroids from non-graphic matroids. Our main result specifies all non-binary orderable matroids. The infinitely many such matroids are all built from $U_{2,n}$ for some $n\geq 4$ by using two operations, which we now describe.

For a matroid $M$ without coloops, a series extension of $M$ is \emph{balanced} if, for some integer $k$ exceeding one, each element of $M$ is replaced by $k$ elements in series. We call $k$ the \emph{order} of the balanced series extension. The second operation is a generalization of the operation of adding an element in parallel to another. A \emph{theta-graph} is a graph consisting of a pair of distinct vertices and three internally disjoint paths between them. Now, let $P$ be a nonempty subset of a series class of a matroid $M$. Fix an element $t$ of $P$, contract $P-t$, and relabel $t$ as $t'$ to obtain $M'$. Let $N$ be the cycle matroid of a theta-graph with series classes $\{t'\}$, $P$, and $P'$, where $\vert P'\vert=\vert P\vert$. Finally, let $M''$ be the 2-sum of $M'$ and $N$ with basepoint $t'$. The operation transforming $M$ into $M''$ is called \emph{parallel-path addition}. The \emph{size} of this addition is $\vert P\vert$; we call $P$ and $P'$ \emph{parallel paths} of $M''$, and say that $M''$ is obtained from $M$ by adding $P'$ \emph{in parallel} to $P$. The following theorem is the main result of the paper.

\begin{theorem}
\label{nonbinchar}
Let $M$ be a connected non-binary matroid. Then $M$ is orderable if and only if it can be obtained from $U_{2,n}$ for some $n\geq 4$ by a sequence of the following operations:
\begin{enumerate}
\item[(i)]{balanced series extension; and}
\item[(ii)]{parallel-path addition.}
\end{enumerate}
\end{theorem}

When we come to consider binary orderable matroids, we encounter considerable difficulty. For example, as we show in the next section, $F_7^\ast$ and $M^\ast(K_5)$ are not orderable, yet each has an orderable series extension. In view of this, it is natural to consider additional conditions that one can add to orderability in order to distinguish graphic matroids within binary matroids. The next theorem gives three equivalent such additional conditions.

\begin{theorem}
\label{exminororderable}
The following are equivalent for a binary matroid $M$:
\begin{enumerate}
\item[(i)]{$M$ is graphic.}
\item[(ii)]{every minor of $M$ is orderable.}
\item[(iii)]{every series minor of $M$ is orderable.}
\item[(iv)]{every parallel minor of $M$ is orderable.}
\end{enumerate}
\end{theorem}

Although, as noted above, there are orderable binary matroids that are not graphic, we know of no counterexample to the following.

\begin{conjecture}
\label{3connbinconj}
A $3$-connected orderable binary matroid is graphic.
\end{conjecture}

We have, however, made the following progress.

\begin{theorem}
\label{4connreg}
A $4$-connected regular orderable matroid is graphic.
\end{theorem}

Another condition one can add to orderability to distinguish graphic matroids within binary matroids involves the theta-graphs in a matroid $M$, where a \emph{theta-graph} in $M$ is a restriction of $M$ that is isomorphic to the cycle matroid of a theta-graph. Equivalently, it is a restriction of $M$ that is isomorphic to a series extension of $U_{1,3}$. The series classes of a theta-graph are called its \emph{theta-arcs}. A subset $B$ of a circuit $C$ is a \emph{block} if there is a listing $b_1,b_2,\dots,b_k$ of the elements of $B$ such that $b_i$ and $b_{i+1}$ are adjacent for all $i$ in $[k-1]$. A consistent ordering of a matroid $M$ is a \emph{theta-ordering} if every theta-arc of every theta-graph of $M$ is a block in the ordering; $M$ is \emph{theta-orderable} if it has a theta-ordering.

Theta-orderability turns out to be equivalent to a concept introduced by Wagner~\cite{wagnerarcs}. For distinct circuits $C$ and $D$ of a matroid $M$, an \emph{arc} of $C$ is a minimal non-empty subset $A$ of $C$ such that $A\cup D$ contains at least two circuits. A set $\{A_1, A_2, A_3\}$ of arcs of a common circuit is \emph{incompatible} if $A_1\cap A_2\cap A_3\neq \emptyset$ and $A_i-(A_j\cup A_k)\neq \emptyset$ for all $i$, $j$, and $k$ such that $\{i,j,k\}=\{1,2,3\}$. In Section~\ref{theta}, we prove the following characterization of theta-orderable binary matroids. The equivalence of (i) and (ii) is Wagner's main result~\cite{wagnerarcs}.

\begin{theorem}
\label{wag}
The following are equivalent for a binary matroid $M$:
\begin{enumerate}
\item[(i)]{$M$ is graphic;}
\item[(ii)]{$M$ has no set of incompatible arcs; and}
\item[(iii)]{$M$ is theta-orderable.}
\end{enumerate}
\end{theorem}

The following characterization of theta-orderable non-binary matroids will also be proved in Section~\ref{theta}.

\begin{theorem}
\label{tononbinchar}
Let $M$ be a connected non-binary matroid. Then $M$ is theta-orderable if and only if $M$ is a parallel extension of a balanced series extension of $U_{2,n}$ for some $n\geq 4$.
\end{theorem}

In Section~\ref{prelims}, after some preliminaries, we prove Theorem~\ref{exminororderable}. The proof of our main result, Theorem~\ref{nonbinchar}, is in Section~\ref{nonbinary}, and Theorem~\ref{4connreg} is proved in Section~\ref{progress}.

\section{Preliminaries}
\label{prelims}

Our first proposition collects some basic properties of orderability. These properties will be used frequently and often implicitly. We omit their straightforward proofs.

\begin{proposition}
\label{basics}
Let $M$ be a matroid.
\begin{enumerate}
\item[(i)]{If $M$ is orderable, then $M\del e$ is orderable for all $e\in E(M)$.}
\item[(ii)]{If $r(M)\leq 2$, then $M$ is orderable.}
\item[(iii)]{$M$ is orderable if and only if the connected components of $M$ are orderable.}
\item[(iv)]{$M$ is orderable if and only if $\si(M)$ is orderable.}
\end{enumerate}
\end{proposition}

Next, we note a partial converse to Proposition~\ref{graphicimpord}.

\begin{proposition}
\label{hamgraph}
If $M$ is an orderable binary matroid with a spanning circuit, then $M$ is graphic.
\end{proposition}

\begin{proof}
Let $C$ be a spanning circuit of $M$ and $e$ be an element in $C$. Fix a consistent ordering of $M$, and take a standard binary representation of $M$ with respect to the basis $C-e$. Now construct a graph $G$ beginning with a cycle having edge set $C$, ordered consistently with the fixed ordering of $M$. Now, for each element $f$ of $E(M)-C$, let $C_f$ be the fundamental circuit of $f$ with respect to $C-e$. Because $C_f-f$ is a block in the ordering, we may add an edge $f$ to $G$ as a chord of $C$ so that it forms a cycle with edge set $C_f$. The result is a graph whose cycle matroid has ground set $E(M)$, has $C-e$ as a basis, and has the same fundamental circuits with respect to this basis as $M$. Since $M$ and $M(G)$ are binary, we deduce that $M=M(G)$.
\end{proof}

We now note a necessary condition for a matroid to be orderable, along with some consequences of this condition.

\begin{proposition}
\label{mickeymouse}
Let $M$ be a simple matroid and $X$ be a subset of $E(M)$ with $\vert X\vert\geq 3$. If there are elements $e$ and $f$ in $E(M)-X$ such that $X\cup e$ and $X\cup f$ are both circuits of $M$, then $M$ is not orderable.
\end{proposition}

\begin{proof}
Assume to the contrary that $M$ has a consistent ordering. Notice that the ordering of $X\cup f$ is obtained from that of $X\cup e$ by replacing $f$ with $e$. Let $a$ and $b$ be the elements in $X$ that are adjacent to $e$. Using strong circuit elimination on $X\cup e$ and $X\cup f$, we obtain a circuit $C\subseteq X\cup \{e,f\}$ containing $e$ but not $a$, and another $C'\subseteq X\cup \{e,f\}$ containing $f$ but not $b$.

As $C$ is not properly contained in either $X\cup e$ or $X\cup f$, it must contain both $e$ and $f$. Further, $M$ is simple, so $C\cap X$ is nonempty. Since $a$ and $b$ are the only elements in $X$ adjacent to $e$ or $f$, it follows that $C=\{e,f,b\}$. By symmetry, $C'=\{e,f,a\}$.

Circuit elimination applied to $C$ and $C'$ now yields a circuit $D$ that does not contain $e$. Then $D\subseteq \{a,b,f\}$. Since $\vert X\vert\geq 3$, it follows that $D$ is a proper subset of $X\cup f$, a contradiction.
\end{proof}

\begin{corollary}
\label{relaxations}
Let $M$ be a matroid of rank at least three and $X$ be a circuit-hyperplane of $M$. If $E(M)-X$ is not a parallel class of $M$, then the matroid obtained from $M$ by relaxing $X$ is not orderable.
\end{corollary}

\begin{corollary}
\label{thelonelywhirl}
The only orderable whirl is $U_{2,4}$.
\end{corollary}

We now prove Theorem~\ref{exminororderable}, whose proof relies on the next lemma and its corollary. The following technical property facilitates the statements of these results. A matroid $M$ has the $(e,f,g)$-\emph{property} if
\begin{itemize}
\item[(i)]{$M$ has a circuit containing $\{e,f,g\}$;}
\item[(ii)]{$e$, $f$, and $g$ are distinct; and}
\item[(iii)]{$M$ has a circuit $D$ containing $f$ but neither $e$ nor $g$ and, with the exception of at most one $d$ in $D$, there is a circuit of $M$ containing $\{e,f,g,d\}$.}
\end{itemize}

\begin{lemma}
\label{efglemma}
If a matroid $M$ has the $(e,f,g)$-property, then $f$ is not adjacent to both $e$ and $g$ in a consistent ordering of $M$.
\end{lemma}

\begin{proof}
Suppose $M$ has the $(e,f,g)$-property and $f$ is adjacent to both $e$ and $g$. Then, in the circuit $D$ of condition (iii), $f$ is adjacent to elements $d_1$ and $d_2$ of $D-f$. But $M$ has a circuit containing $\{e,f,g,d_i\}$ for some $i$ in $\{1,2\}$, a contradiction.
\end{proof}

\begin{corollary}
\label{efgcor}
Let $C$ be a circuit of a matroid $M$. Suppose there is an element $c$ of $C$ so that $M$ has the $(e,c,g)$-property for every choice of $e$ and $g$ in $C-c$. Then $M$ does not have a consistent ordering.
\end{corollary}

\begin{proof}[Proof of Theorem~\ref{exminororderable}]
Since graphic matroids are orderable and the class of graphic matroids is minor-closed, (i) implies (ii)-(iv). Let $\mathcal{S}$ be the set
\begin{equation*}
\{F_7, F_7^\ast, M^\ast(K_5), M^\ast(K_{3,3}), M^\ast(K_{3,3}'), M^\ast(K_{3,3}''), M^\ast(K_{3,3}'''), R_{10}\}.
\end{equation*}
By results of Tutte~\cite{tuttekurwag} and Bixby~\cite{bixbyreg,bixbykurwag}, $\mathcal{S}$ contains all binary matroids that are excluded minors, excluded series minors, or excluded parallel minors for the class of graphic matroids. Thus we can prove that (i) follows from each of (ii)-(iv) by showing that none of the matroids in $\mathcal{S}$ is orderable.

\begin{figure}
\centering

\begin{tikzpicture}[style=thick]

\foreach \x in {0,1,2} \draw[fill=black] (90+120*\x:2) circle (0.1);
\foreach \x in {0,1,2} \draw[fill=black] (-90+120*\x:1) circle (0.1);
\draw[fill=black] (0,0) circle (0.1);

\foreach \x in {0,1,2} \draw (90+120*\x:2)--(-90+120*\x:1) (90+120*\x:2)--(210+120*\x:2);
\draw (150:1).. controls (-90:1.5) ..(30:1);

\draw (60:0.4) node {$3$};
\draw (90:2.3) node {$1$};
\draw (210:2.3) node {$5$};
\draw (330:2.3) node {$7$};
\draw (150:1.3) node {$2$};
\draw (-90:1.3) node {$6$};
\draw (30:1.3) node {$4$};

\end{tikzpicture}
\caption{The matroid $F_7$.}
\label{f7}
\end{figure}

Let $F_7$ be labelled as in Figure~\ref{f7}. Using the element 1 in the circuit $\{1,2,3,4\}$, Corollary~\ref{efgcor} gives that $F_7$ has no consistent ordering.

\begin{figure}[b]
\centering
\begin{tikzpicture}[style=thick]

\draw[fill=black] (-3.5,2.5) circle (0.1);
\draw (-3.55,2.5) node[left] {1};
\draw[fill=black] (-1.75,3.75) circle (0.1);
\draw (-1.75,3.75) node[above left] {2};
\draw[fill=black] (1.75,3.75) circle (0.1);
\draw (1.75,3.75) node[above right] {3};
\draw[fill=black] (3.5,2.5) circle (0.1);
\draw (3.55,2.5) node[right] {4};
\draw[fill=black] (-1.75,2.25) circle (0.1);
\draw (-1.75,2.25) node [below left] {5};
\draw[fill=black] (-1.15,3.16) circle (0.1);
\draw (-1.2,3.55) node {6};
\draw[fill=black] (1.15,3.16) circle (0.1);
\draw (1.25,2.8) node {7};

\draw (-5,4.5)--(0,5.5)--(5,4.5)--(5,0.5)--(0,1.5)--(-5,0.5)--(-5,4.5);
\draw (0,1.5)--(0,5.5);
\draw (-3.5,2.5)--(0,5) (-3.5,2.5)--(0,2);
\draw (3.5,2.5)--(0,5);
\draw (-3.5,2.5)--(0,3.5)--(3.5,2.5);
\draw (1.75,3.75)--(0,2)--(-1.75,3.75);
\draw (-1.75,2.25)--(0,5);
\draw (-1.75,3.75).. controls (-2,1.75) ..(0,3.5);

\end{tikzpicture}
\caption{The matroid $F_7^\ast$.}
\label{f7star}
\end{figure}

Let $F_7^\ast$ be labelled as in Figure~\ref{f7star}. Consider the circuits $C_1=\{1,2,3,4\}$, $C_2=\{1,3,5,7\}$, and $C_3=\{2,4,5,7\}$.
The ordering of a four-element circuit is uniquely determined by a single pair of non-adjacent elements, and the automorphism group of $F_7^\ast$ is doubly transitive. Thus we may assume that $C_1$ has the ordering $(1\ 2\ 3\ 4)$.

Since 1 and 3 are not adjacent in $C_1$, it follows that $C_2$ has the ordering $(1\ 5\ 3\ 7)$. Thus 5 and 7 are non-adjacent, so $C_3$ has the ordering $(2\ 5\ 4\ 7)$. However, the elements of the set $\{1,2,5\}$ are now pairwise adjacent, so the circuit $\{1, 2, 5, 6\}$ cannot be ordered. Thus $F_7^\ast$ has no consistent ordering.

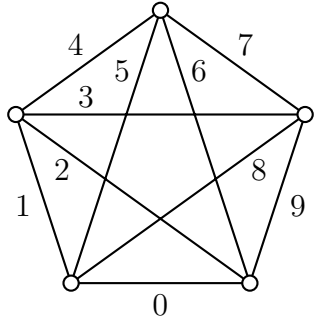
\begin{figure}
\centering
\begin{tikzpicture}[style=thick]

\foreach \x in {0,1,2,3,4} \draw (90+72*\x:2)--(162+72*\x:2);
\draw (90:2)--(234:2)--(18:2)--(162:2)--(306:2)--(90:2);

\foreach \x in {0,1,2,3,4} \draw[fill=white] (90 + 72*\x:2) circle (0.1);

\draw (-90:1.9) node {0};
\draw (198:1.9) node {1};
\draw (185:1.3) node {2};
\draw (139:1.3) node {3};
\draw (126:1.9) node {4};
\draw (113:1.3) node {5};
\draw (67:1.3) node {6};
\draw (54:1.9) node {7};
\draw (-5:1.3) node {8};
\draw (-18:1.9) node {9};

\end{tikzpicture}
\caption{The graph $K_5$.}
\label{k5}
\end{figure}

Let $M^\ast(K_5)$ be labelled as in Figure~\ref{k5}, and assume that $M^\ast(K_5)$ has a consistent ordering. Let $C$ be the circuit $\{1,2,3,4\}$. By symmetry, we may assume its ordering is $(1\ 2\ 3\ 4)$. This ordering and the circuit $\{1,2,4,7,8,9\}$ give that 1 and 8 are not adjacent, so the circuit $\{0,1,5,8\}$ must be ordered $(0\ 1\ 5\ 8)$. Similarly, the circuit $\{1,2,3,5,6,7\}$ gives that 2 and 6 are not adjacent, so $\{0,2,6,9\}$ must be ordered $(0\ 2\ 9\ 6)$. Now 0 is adjacent to 1, 6, and 8 in the circuit $\{0,1,4,6,7,8\}$, a contradiction.

\begin{figure}[h]
\centering
\begin{tikzpicture}[style=thick]

\foreach \x in {0,2,4}
\foreach \z in {0,2,4} \draw (\x,2)--(\z,0);

\foreach \x in {0,2,4}
\foreach \y in {0,2} \draw[fill=white] (\x,\y) circle (0.1);

\draw (-0.2,1.6) node {1};
\draw (0.3,1.4) node {2};
\draw (0.7,1.9) node {3};
\draw (1.5,1.8) node {4};
\draw (1.8,1.4) node {5};
\draw (2.5,1.8) node {6};
\draw (3.3,1.9) node {7};
\draw (3.7,1.4) node {8};
\draw (4.2,1.6) node {9};

\end{tikzpicture}
\caption{The graph $K_{3,3}$.}
\label{k33}
\end{figure}

Let $M^\ast(K_{3,3})$ be labelled as in Figure~\ref{k33}. We shall use Corollary~\ref{efgcor} letting $c$ be the element 1 in the circuit $C=\{1,3,5,8\}$ of $M^\ast(K_{3,3})$. The cases $\{e,g\}=\{3,5\}$ and $\{e,g\}=\{3,8\}$ are symmetric, and the circuits $C$ and $\{1,3,5,7,9\}$ certify that $M$ has the $(3,1,5)$-property with $D=\{1,4,8,9\}$. The circuit $\{1,5,6,8,9\}$ certifies that $M$ has the $(5,1,8)$-property with $D=\{1,2,6,9\}$. Corollary~\ref{efgcor} now implies that $M^\ast(K_{3,3})$ has no consistent ordering.

\begin{figure}[h]
\centering
\begin{tikzpicture}[style=thick]

\foreach \x in {0,2,4}
\foreach \z in {0,2,4} \draw (\x,2)--(\z,0);
\draw (0,0)--(2,0);

\foreach \x in {0,2,4}
\foreach \y in {0,2} \draw[fill=white] (\x,\y) circle (0.1);

\draw (-0.2,1.6) node {1};
\draw (0.3,1.4) node {2};
\draw (0.7,1.9) node {3};
\draw (1.5,1.8) node {4};
\draw (1.8,1.4) node {5};
\draw (2.5,1.8) node {6};
\draw (3.3,1.9) node {7};
\draw (3.7,1.4) node {8};
\draw (4.2,1.6) node {9};
\draw (1,-0.25) node {$a$};

\end{tikzpicture}
\caption{The graph $K_{3,3}'$.}
\label{k33'}
\end{figure}

The next two cases will also use Corollary~\ref{efgcor}. Let $M^\ast(K_{3,3}')$ be labelled as in Figure~\ref{k33'}, and let $c$ be the element 3 in the circuit $C=\{3,6,7,8\}$ of $M^\ast(K_{3,3}')$. The cases $\{e,g\}=\{6,7\}$ and $\{e,g\}=\{6,8\}$ are symmetric, and the circuit $\{2,3,5,6,7,8\}$ certifies that $M$ has the $(6,3,7)$-property with $D=\{1,2,3\}$. The circuit $C$ certifies that $M$ has the $(7,3,8)$-property with $D=\{3,6,9\}$. Corollary~\ref{efgcor} now implies that $M^\ast(K_{3,3}')$ has no consistent ordering.

\begin{figure}[ht]
\centering
\begin{tikzpicture}[style=thick]

\foreach \x in {0,2,4}
\foreach \z in {0,2,4} \draw (\x,2)--(\z,0);
\draw (0,0)--(2,0)--(4,0);

\foreach \x in {0,2,4}
\foreach \y in {0,2} \draw[fill=white] (\x,\y) circle (0.1);

\draw (-0.2,1.6) node {1};
\draw (0.3,1.4) node {2};
\draw (0.7,1.9) node {3};
\draw (1.5,1.8) node {4};
\draw (1.8,1.4) node {5};
\draw (2.5,1.8) node {6};
\draw (3.3,1.9) node {7};
\draw (3.7,1.4) node {8};
\draw (4.2,1.6) node {9};
\draw (1,-0.25) node {$a$};
\draw (3,-0.25) node {$b$};

\end{tikzpicture}
\caption{The graph $K_{3,3}''$.}
\label{k33''}
\end{figure}

Let $M^\ast(K_{3,3}'')$ be labelled as in Figure~\ref{k33''}, and let $c$ be the element 1 in the circuit $C=\{1,3,5,8,a,b\}$ of $M^\ast(K_{3,3}'')$. When $\{e,g\}\subseteq C-3$, the circuit $C$ certifies that $M$ has the $(e,1,g)$-property with $D=\{1,2,3\}$. Each of the remaining cases uses $D=\{1,4,7,a\}$. The cases $\{e,g\}=\{3,5\}$ and $\{e,g\}=\{3,8\}$ are symmetric, and the circuit $\{1,3,5,7,9,a,b\}$ certifies that $M$ has the $(3,1,5)$ property. The circuits $\{1,3,4,6,8,a,b\}$ and $\{1,3,5,7,9,a,b\}$ certify the $(3,1,a)$-property. Finally, the circuit $\{1,3,4,6,8,a,b\}$ certifies the $(3,1,b)$-property. Corollary~\ref{efgcor} now implies that $M^\ast(K_{3,3}'')$ has no consistent ordering.

\begin{figure}[h]
\centering
\begin{tikzpicture}[style=thick]

\foreach \x in {0,2,4}
\foreach \z in {0,2,4} \draw (\x,2)--(\z,0);
\draw (0,0)--(2,0)--(4,0);
\draw (0,0).. controls (2,-0.6) ..(4,0);

\foreach \x in {0,2,4}
\foreach \y in {0,2} \draw[fill=white] (\x,\y) circle (0.1);

\draw (-0.2,1.6) node {1};
\draw (0.3,1.4) node {2};
\draw (0.7,1.9) node {3};
\draw (1.5,1.8) node {4};
\draw (1.8,1.4) node {5};
\draw (2.5,1.8) node {6};
\draw (3.3,1.9) node {7};
\draw (3.7,1.4) node {8};
\draw (4.2,1.6) node {9};
\draw (1.2,0.2) node {$a$};
\draw (2.8,0.2) node {$b$};
\draw (2,-0.7) node {$c$};

\end{tikzpicture}
\caption{The graph $K_{3,3}'''$.}
\label{k33'''}
\end{figure}

Let $M^\ast(K_{3,3}''')$ be labelled as in Figure~\ref{k33'''}. We begin by noting that there must be at least one adjacent pair in the set $\{1,4,7\}$ due to the circuit $\{1,4,7,a,c\}$. By symmetry, we may assume that 1 and 4 are adjacent.

Combining this adjacent pair with the three-element circuits, we get that 2145 is a block in the circuit $\{1,2,4,5,9,b,c\}$. Therefore 4 is not adjacent to 9, $b$, or $c$. This means that, in the circuit $\{3,4,5,9,b,c\}$, we must have 4 adjacent to 3. Using the three-element circuit $\{4,5,6\}$, we now have that 4 is adjacent to 1, 3, and 6. Therefore the circuit $\{1,3,4,6,8,a,b\}$ cannot be ordered consistently, and $M^\ast(K_{3,3}''')$ has no consistent ordering.

\begin{figure}[h]
\centering
\begin{tikzpicture}[style=thick]

\foreach \x in {0,2,4}
\foreach \y in {0,2} \draw (\x-0.15,\y-0.15) rectangle ++(0.3,0.3);

\foreach \x in {0,2,4}
\foreach \z in {0,2,4} \draw (\x,2)--(\z,0);

\foreach \x in {0,2,4}
\foreach \y in {0,2} \draw[fill=white] (\x,\y) circle (0.1);

\draw (-0.2,1.6) node {1};
\draw (0.3,1.4) node {2};
\draw (0.7,1.9) node {3};
\draw (1.5,1.8) node {4};
\draw (1.8,1.4) node {5};
\draw (2.5,1.8) node {6};
\draw (3.3,1.9) node {7};
\draw (3.7,1.4) node {8};
\draw (4.2,1.6) node {9};

\end{tikzpicture}
\caption{A graft corresponding to $R_{10}$.}
\label{k33graft}
\end{figure}
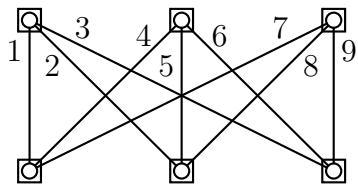

Let $M$ be the graft matroid of $K_{3,3}$ where the graft element $e_{\gamma}$ corresponds to the set of boxed vertices in Figure~\ref{k33graft}. Then $M\cong R_{10}$. Using Corollary~\ref{efgcor} again, let $c$ be the element 1 in the circuit $C=\{1,2,4,5\}$ of $M$. When $\{e,g\}=\{2,4\}$, the circuit $\{1,2,4,6,8,9\}$ certifies the $(2,1,4)$-property when $D=\{1,3,4,6\}$. When $\{e,g\}=\{2,5\}$, the circuit $\{1,2,5,6,7,9\}$ certifies the $(2,1,5)$-property with $D=\{1,3,7,9\}$. Finally, when $\{e,g\}=\{4,5\}$, the circuit $\{1,4,5,6,7,e_\gamma\}$ certifies the $(4,1,5)$-property with $D=\{1,6,8,e_\gamma\}$. Corollary~\ref{efgcor} now implies that $R_{10}$ has no consistent ordering.
\end{proof}

We conclude this section with a pair of examples that indicate the potential difficulty of characterizing orderable binary matroids.

\begin{example}
\label{nonordO1}
This example describes a 12-element orderable series extension of $F_7^\ast$, which we refer to as $O_1$. Thus, the pair $O_1$ and $F_7^\ast$ demonstrates that the class of binary orderable matroids is not closed under the taking of series minors. Let $F_7^\ast$ be labelled as in Figure~\ref{f7star}. We obtain $O_1$ by adding $1'$, $2'$, and $7'$ in series with 1, 2, and 7, respectively, and adding $4'$ and $4''$ in series with 4. Figure~\ref{O1ordering} gives a consistent ordering of the circuits of $O_1$.

\begin{figure}[h]
\centering

$(1\ 5\ 1'\ 2'\ 6\ 2)$\hspace{0.5cm}
$(1\ 5\ 1'\ 7'\ 3\ 7)$\hspace{0.5cm}
$(2\ 6\ 2'\ 7'\ 3\ 7)$\hspace{0.5cm}
$(3\ 4\ 5\ 4'\ 6\ 4'')$\\
\ \\
$(1\ 4'\ 2'\ 1'\ 4\ 3\ 4''\ 2)$\hspace{0.5cm}
$(1\ 7\ 4\ 1'\ 7'\ 4''\ 6\ 4')$\hspace{0.5cm}
$(2\ 7\ 4\ 5\ 4'\ 2'\ 7'\ 4'')$\\

\, 

\caption{A consistent ordering of $O_1$.}
\label{O1ordering}
\end{figure}
\end{example}

\begin{example}
\label{nonordO2}
Let $K_5$ be labelled as in Figure~\ref{k5}. We obtain a regular, non-graphic matroid $O_2$ from $M^\ast(K_5)$ by adding elements $0'$ and $2'$ in series with 0 and 2, respectively. Figure~\ref{O2ordering} gives a consistent ordering of $O_2$.

\begin{figure}[h]
\centering

$(4\ 6\ 5\ 7)$\hspace{0.5cm}
$(2'\ 1\ 2\ 6\ 5\ 8\ 9)$\hspace{0.5cm}
$(0'\ 1\ 0\ 9\ 3\ 4\ 6)$\hspace{0.5cm}
$(2'\ 1\ 2\ 4\ 7\ 8\ 9)$\\
\ \\

$(3\ 7\ 8\ 9)$\hspace{0.5cm}
$(2\ 1\ 2'\ 3\ 7\ 5\ 6)$\hspace{0.5cm}
$(0'\ 1\ 0\ 8\ 7\ 4\ 6)$\hspace{0.5cm}
$(0'\ 1\ 0\ 9\ 3\ 7\ 5)$\\
\ \\

$(0'\ 2'\ 3\ 4\ 2\ 0\ 8\ 5)$\hspace{0.5cm}
$(0'\ 2'\ 9\ 0\ 2\ 4\ 7\ 5)$\hspace{0.5cm}
$(2'\ 0'\ 6\ 2\ 0\ 8\ 7\ 3)$\\
\ \\

$(2\ 1\ 2'\ 3\ 4)$\hspace{0.5cm}
$(0\ 1\ 0'\ 5\ 8)$\hspace{0.5cm}
$(0\ 2\ 6\ 0'\ 2'\ 9)$\hspace{0.5cm}
$(3\ 4\ 6\ 5\ 8\ 9)$\\
\,

\caption{A consistent ordering of $O_2$.}
\label{O2ordering}
\end{figure}
\end{example}

\section{A Characterization of Non-Binary Orderable Matroids}
\label{nonbinary}

In this section, we prove Theorem~\ref{nonbinchar}. We begin by finding the orderable series extensions of uniform matroids and their consistent orderings. These results allow us to characterize the non-binary orderable matroids that are 3-connected, from which we obtain the full characterization using the canonical tree decomposition of Cunningham and Edmonds~\cite{cunned}.

A uniform matroid is binary if and only if it is graphic. Thus, the binary uniform matroids are certainly orderable, as are those whose rank is at most two. Proposition~\ref{mickeymouse} implies this list is complete.

\begin{corollary}
\label{baseuniformprop}
A uniform matroid is orderable if and only if it is binary or has rank at most two.
\end{corollary}

The next two results deduce the structure of a consistent ordering of a series extension of a non-binary uniform matroid, and show that such an ordering can be used to consistently order the underlying uniform matroid. For a non-coloop element $e$ of a matroid $M$, we denote the series class of $M$ containing $e$ by $S_e$ or sometimes by $S_e(M)$.

Let $M$ be a matroid with a consistent ordering. Suppose $X$ and $Y$ are disjoint subsets of a circuit $C$ of $M$. We say $X$ and $Y$ are \emph{adjacent} if there is an adjacent pair of elements $x$ and $y$, where $x$ belongs to $X$ and $y$ belongs to $Y$. Let $\mathcal{B}$ be the union of a set of blocks that belong to a common circuit of $M$. If there is a listing $B_1,B_2,\dots,B_k$ of the blocks in $\mathcal{B}$ such that $B_i$ and $B_{i+1}$ are adjacent for all $i$ in $[k-1]$, then $\mathcal{B}$ is a \emph{section}. Finally, let $S$ be a series class of $M$. If a block of $M$ is contained in $S$ and is maximal with this property, then it is called an \emph{$S$-block}.

\begin{lemma}
\label{uniformkeylemma}
Let $M$ be an orderable series extension of a non-binary uniform matroid $U_{r,n}$ and fix a consistent ordering of $M$. Let $C$ be a circuit of $M$, and let $x$ and $y$ be elements of $C$ from distinct series classes of $M$.
\begin{enumerate}
\item[(i)]{If a section $K$ in $C$ is adjacent to a pair of $S_x$-blocks, then $K$ must contain an $S_y$-block.}
\item[(ii)]{Every series class $S$ of $M$ has the same number of $S$-blocks.}
\end{enumerate}
\end{lemma}

\begin{proof}
For (i), suppose to the contrary that there is a section $K$ in $C$ that contains no $S_y$-block and is adjacent to a pair of distinct $S_x$-blocks. As $M$ is non-binary, $2\leq r\leq n-2$ and there is a circuit $D_x$ of $M$ that contains $K$ and $S_x$ but avoids $S_y$. Let $D_y=(D_x- S_x)\cup S_y$. Observe that, since $M$ is a series extension of $U_{r,n}$, the set $D_y$ is a circuit. The consistency of $D_y$ with $C$ implies that $K$ is not adjacent to $S_y$-blocks in $D_y$, but the consistency of $D_y$ with $D_x$ gives that $K$ can only be adjacent to $S_y$-blocks in $D_y$, a contradiction.

We now deduce (ii) from (i). Let $S$ be a series class of $E(M)$ for which the number of $S$-blocks is as large as possible. We may assume this number exceeds one. In a circuit $C$ containing $S$, let $K$ be a minimal section that is adjacent to a pair of distinct $S$-blocks. Note that the number of such minimal sections in $C$ equals the number of $S$-blocks. Let $S'$ be a series class of $M$ contained in $C$ that is distinct from $S$. Part (i) implies there is an $S'$-block in $K$ and, as $K$ contains no $S$-blocks, (i) further implies that there is exactly one $S'$-block in $K$. Thus there are the same number of $S'$-blocks as $S$-blocks. Part (ii) now follows.
\end{proof}

\begin{proposition}
\label{downwarduniformprop}
Let $U_{r,n}$ be a non-binary uniform matroid. If a series extension of $U_{r,n}$ is orderable, then so is $U_{r,n}$.
\end{proposition}

\begin{proof}
Let $M$ be an orderable series extension of $U_{r,n}$ and fix a consistent ordering of $M$. By Lemma~\ref{uniformkeylemma}(ii), there is an integer $k\geq 1$ such that every series class of $M$ is divided into exactly $k$ blocks. If $k=1$, the result follows immediately, so assume $k\geq 2$.

Let $[n]$ be the ground set of $U_{r,n}$. Consider the circuit $C$ of $M$ that contains the set $\{1,2,\dots,r+1\}$. Label the $S_1$-blocks in $C$ as $B_1,B_2,\dots,B_k$, such that $B_i$ and $B_{i+1}$ abut a section $K_i$ that does not contain $S_1$-blocks, as in Figure~\ref{chunkdiagram}.

\begin{figure}
\centering
\begin{tikzpicture}[style=thick]

\draw (0,0) circle (2);

\draw (90:2.3) node {$B_1$};
\draw (105:1.9)--(105:2.1) (75:1.9)--(75:2.1);

\draw (45:2.3) node {$K_1$};

\draw (0:2.3) node {$B_2$};
\draw (15:1.9)--(15:2.1) (-15:1.9)--(-15:2.1);

\draw (-45:2.3) node {$K_2$};

\draw (-90:2.3) node {$B_3$};
\draw (-105:1.9)--(-105:2.1) (-75:1.9)--(-75:2.1);

\draw (225:2.3) node {$\ddots$};

\draw (180:2.3) node {$B_k$};
\draw (195:1.9)--(195:2.1) (165:1.9)--(165:2.1);

\draw (135:2.3) node {$K_k$};

\end{tikzpicture}
\caption{The circuit $C$ in the proof of Proposition~\ref{downwarduniformprop}.}
\label{chunkdiagram}
\end{figure}
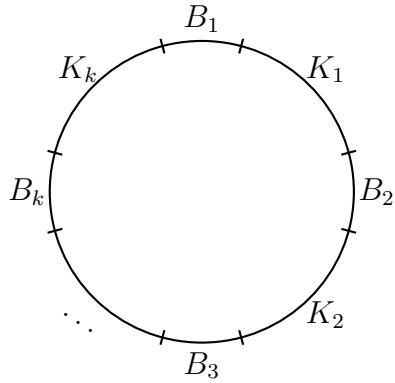

Applying Lemma~\ref{uniformkeylemma}(i), we see that each section $K_i$ contains exactly one $S_j$-block for all $j$ in $\{2,3,\dots,r+1\}$. Thus, $B_i\cup K_i$ defines a permutation of $\{1,2,\dots,r+1\}$ that begins with 1. We show this permutation is the same for all $i$.

Without loss of generality, suppose the block in $K_1$ adjacent to $B_1$ is an $S_2$-block. If the block in $K_2$ adjacent to $B_2$ is an $S_j$-block with $j\neq 2$, then the $S_j$-blocks in $K_1$ and $K_2$ abut a section that contains no $S_2$-block, contradicting Lemma~\ref{uniformkeylemma}(i). Thus the block in $K_2$ adjacent to $B_2$ is an $S_2$-block. Repeating this argument gives that $B_1\cup K_1$ and $B_2\cup K_2$ define the same permutation on $\{1,2,\dots,r+1\}$. It follows that $B_i\cup K_i$ defines the same permutation on $\{1,2,\dots,r+1\}$ for all $i$ in $[n]$. Thus $B_i\cup K_i\cup B_{i+1}$ defines the same reversible cyclic ordering on $\{1,2,\dots,r+1\}$ for all $i$ in $[n]$; it is this reversible cyclic ordering that we extract from $C$ and use to order the circuit $\{1,2,\dots,r+1\}$ in $U_{r,n}$.

In this way, every circuit of $U_{r,n}$ is ordered using the corresponding circuit of $M$. Since the ordering of $M$ is consistent, so too is the ordering it gives to $U_{r,n}$.
\end{proof}

\begin{theorem}
\label{exclunifserminor}
Let $U_{r,n}$ be a non-binary uniform matroid of rank at least three. If $M$ is a matroid with a series minor isomorphic to $U_{r,n}$, then $M$ is not orderable.
\end{theorem}

\begin{proof}
By \cite[Proposition 5.4.2]{oxley}, we may write $U_{r,n}=M\del X/Y$ where each element of $Y$ is in series with an element of $M\del X$ not in $Y$. By Corollary~\ref{baseuniformprop}, the matroid $U_{r,n}$ is not orderable. Therefore, by Proposition~\ref{downwarduniformprop}, neither is its series extension $M\del X$. Thus, $M$ is not orderable.
\end{proof}

Recall that, in a balanced series extension $N$ of a matroid $M$ without coloops, each element of $M$ is replaced by $k$ elements in series for some positive integer $k$.

\begin{lemma}
\label{balancebringsorder}
Let $N$ be a balanced series extension of $U_{2,n}$ for some $n\geq 4$. Then $N$ is orderable.
\end{lemma}

\begin{proof}
Let $[n]$ be the ground set of $U_{2,n}$. For each $x$ in $E(U_{2,n})$, let $\{x_0, x_1, \dots,x_{k-1}\}$ be the series class $S_x$ of $N$ that corresponds to $x$. Subscript arithmetic will be done modulo $k$. Let $\mathcal{C}_1$ be the set of circuits of $N$ containing $S_1$. For the circuit $C=S_1\cup S_x\cup S_y$ in $\mathcal{C}_1$ with $x< y$, give $C$ the ordering $$(y_0\ 1_0\ x_0\ y_1\ 1_1\ x_1\ \dots\ y_{k-1}\ 1_{k-1}\ x_{k-1}).$$ To see that the circuits in $\mathcal{C}_1$ are consistent, suppose $e$ and $f$ belong to common circuits in $\mathcal{C}_1$. Then at least one of $e$ and $f$, say $e$, is in $S_1$. If $f$ is in $S_1$, then $e$ and $f$ are never adjacent. Otherwise, $f$ is in $S_z$ for some $z>1$, so $e=1_s$ and $f=z_t$ for some $s$ and $t$. If $s=t$, then $e$ and $f$ are always adjacent; if $s\neq t$, then $e$ and $f$ are never adjacent.

Now let $\mathcal{C}_2$ be the set of circuits of $N$ not containing $S_1$. For the circuit $D=S_x\cup S_y\cup S_z$ in $\mathcal{C}_2$ with $1< x<y<z$, give $D$ the ordering $$(z_1\ x_0\ y_1\ z_2\ x_1\ y_2\ ...\ z_0\ x_{k-1}\ y_0).$$ Note that $x_i$ is always adjacent to $y_{i+1}$ and $z_{i+1}$. Further, the blocks $z_{i}x_{i-1}y_{i}$ are ordered so that $y_i$ is always adjacent to $z_{i+1}$. Thus, the circuits in $\mathcal{C}_2$ are consistent with those in $\mathcal{C}_1$.

Finally, the circuits in $\mathcal{C}_2$ are consistent. Suppose instead that there are circuits $C$ and $D$ in $\mathcal{C}_2$ and elements $x_s$ and $y_t$ in $C\cap D$ so that $x_s$ and $y_t$ are adjacent in $C$ but not in $D$. Assume $x<y$. From $C$, we have that $t=s+1$, but, from $D$, we have that $t\neq s+1$, a contradiction.
\end{proof}

The following proposition specializes some of the results about uniform matroids to $U_{2,n}$ with $n\geq 4$. These rank-two uniform matroids will serve as the foundation from which all non-binary orderable matroids are built.

\begin{proposition}
\label{ordimpbalanced}
Let $M$ be an orderable series extension of $U_{2,n}$ for some $n\geq 4$, and fix a consistent ordering of $M$. Then
\begin{enumerate}
\item[(i)]{for all series classes $S$ of $M$, every $S$-block of the ordering consists of a single element; and}
\item[(ii)]{$M$ is a balanced series extension of $U_{2,n}$.}
\end{enumerate}
\end{proposition}

\begin{proof}
Statement (ii) follows from combining (i) with Lemma~\ref{uniformkeylemma}(ii), so it suffices to show (i). Let $E(U_{2,n})=[n]$. Suppose, to the contrary, that $M$ has an $S_1$-block $B$ of size at least two.

Applying Lemma~\ref{uniformkeylemma}(i), we have that $B$ is adjacent to both an $S_2$-block and an $S_3$-block in the circuit of $M$ containing $\{1,2,3\}$. Let $1_2$ be the element of $B$ adjacent to the $S_2$-block and let $1_3$ be the element of $B$ adjacent to the $S_3$-block, where $1_2$ and $1_3$ are necessarily distinct. In the circuit of $M$ containing $\{1,2,4\}$, Lemma~\ref{uniformkeylemma}(i) now gives that $B$ is adjacent to both an $S_2$-block and an $S_4$-block. Consistency dictates that $1_2$ is again adjacent to the $S_2$-block. Therefore $1_3$ is now adjacent to the $S_4$-block.

Now consider the circuit of $M$ containing $\{1,3,4\}$. Consistency with the two aforementioned circuits requires that $1_3$ be adjacent to both an $S_3$-block and an $S_4$-block. As $\vert B\vert \geq 2$, this is a contradiction.
\end{proof}

The next theorem identifies all orderable matroids that are 3-connected and non-binary.

\begin{theorem}
\label{nonbin3conn}
If $M$ is a $3$-connected non-binary orderable matroid, then $M\cong U_{2,n}$ for some $n\geq 4$.
\end{theorem}

The next two results will be used in the proof of this theorem.

\begin{proposition}
\label{noU35minor}
If $M$ is an orderable matroid, then $M$ has no minor isomorphic to $U_{3,5}$.
\end{proposition}

\begin{proof}
Assume instead that $M\del X/Y\cong U_{3,5}$, with $X$ coindependent and $Y$ independent. Then $M^\ast/ X\del Y\cong U_{2,5}$ where $M^\ast /X$ has rank two. Thus, after deleting a set $Z$ of loops from $M^\ast/X$, we obtain a parallel extension of $U_{2,n}$ for some $n\geq 5$. This makes $M\del (X\cup Z)$ an orderable series extension of $U_{n-2,n}$, contradicting Theorem~\ref{exclunifserminor}.
\end{proof}

\begin{proposition}
\label{no3whirlminor}
If $M$ is an orderable matroid, then $M$ has no minor isomorphic to $\mathcal{W}^3$.
\end{proposition}

The proof of this proposition will rely on the next lemma and its corollary. This second pair of results will use the following modification of the $(e,f,g)$-property. A matroid $M$ has the \emph{series} $(e,f,g)$-\emph{property} if
\begin{itemize}
\item[(i)]{$M$ has a circuit containing $\{e,f,g\}$;}
\item[(ii)]{$S_f(M)$ is distinct from both $S_e(M)$ and $S_g(M)$; and}
\item[(iii)]{$M$ has a circuit $D$ containing $f$ but not $\{e,g\}$ and, for each $d$ in $D$, there is a circuit of $M$ containing $\{e,f,g,d\}$.}
\end{itemize}
Note that $e$ and $g$ may be equal in this definition.

\begin{lemma}
\label{seriesefglemma}
Suppose that $M$ has the series $(e,f,g)$-property and that $N$ is a series extension of $M$. Then, in a consistent ordering of $N$, if $S_e(N)\neq S_g(N)$, then no $S_f(N)$-block is adjacent to both an $S_e(N)$-block and an $S_g(N)$-block; and, if $S_e(N)=S_g(N)$, then no $S_f(N)$-block is adjacent to two $S_e(N)$-blocks.
\end{lemma}

\begin{proof}
Let $D$ be the circuit of $M$ whose existence is guaranteed by condition (iii). Let $D'$ be the circuit of $N$ corresponding to $D$, and let $B_f$ be an $S_f(N)$-block. Notice $D$ must have an element $d$ not in $\{e,f,g\}$, so $D'-(S_e(N)\cup S_f(N)\cup S_g(N))$ is nonempty. If $S_e(N)=S_g(N)$ and $B_f$ is adjacent to two $S_e(N)$-blocks, then $e$ is not in $D$, so $B_f$ is not adjacent to any elements of $D'-B_f$, a contradiction. Now suppose $S_e(N)\neq S_g(N)$ and, without loss of generality, suppose $e$ is in $D$ but $g$ is not. If $B_f$ is adjacent to an $S_e(N)$-block and an $S_g(N)$-block, then all of the elements in $D'-B_f$ adjacent to $B_f$ are in $S_e(N)$. This contradicts the fact that $B_f$ is adjacent to an $S_e(N)$-block and an $S_g(N)$-block in a common circuit.
\end{proof}

\begin{corollary}
\label{seriesefgcor}
Let $C$ be a circuit of a matroid $M$. Suppose that $C$ contains an element $c$ so that $M$ has the series $(e,c,g)$-property for every choice of $e$ and $g$ in $C-c$. Then no series extension of $M$ is orderable.
\end{corollary}

\begin{proof}[Proof of Proposition~\ref{no3whirlminor}]
Assume instead that $M\del X/Y\cong\mathcal{W}^3$, with $X$ coindependent and $Y$ independent. Let $L$ be the set of loops of $M^\ast/X$, and let $N$ denote $M^\ast/(X\cup L)$. Note that $N$ is a loopless rank-3 extension of $\mathcal{W}^3$, so $\si(N)$ is $3$-connected. Further, $N$ is a parallel extension of $\si(N)$, which makes $N^\ast$ an orderable series extension of $\co(N^\ast)$.

\begin{sublemma}
\label{ternarysublemma}
$\si(N)$ is ternary.
\end{sublemma}

To see this, first note that, as $N^\ast$ is orderable, it has no $U_{3,5}$-minor by Proposition~\ref{noU35minor}. Thus, $\si(N)$ has no $U_{2,5}$-minor. As $\si(N)$ is $3$-connected and its rank and corank each exceed two,~\cite[Proposition 12.2.15]{oxley} gives that $\si(N)$ has no $U_{3,5}$-minor. The rank of $F_7^\ast$ exceeds three, so $\si(N)$ also has no $F_7^\ast$-minor.

Finally, suppose $\si(N)$ has an $F_7$-minor. Then $\si(N)\vert Z\cong F_7$ for some set $Z$. As $F_7$ has no $\mathcal{W}^3$-minor, $\si(N)$ has an element $e$ not in $Z$. Then $\si(N)/e$ has a $U_{2,5}$-restriction, a contradiction. We conclude that $\si(N)$ has no $F_7$-minor. Thus~\ref{ternarysublemma} holds.

By~\ref{ternarysublemma}, $\si(N)$ has the form $PG(2,3)-K$, where $K$ is a restriction of $O_7$, the complement of $\mathcal{W}^3$ in $PG(2,3)$. The matroid $O_7$ is obtained from $M(K_4)$ by adding a point freely to an existing 3-point line; the fifteen restrictions of $O_7$ are given in Figure~\ref{o7restrictions}. In the remainder of the proof, we eliminate each possibility for $K$.

\begin{figure}
\centering
\begin{tabular}{|c|l|}
\hline
Number $n$ of elements & Restrictions of $O_7$ with $n$ elements\\
\hline
0 & $U_{0,0}$\\
\hline
1 & $U_{1,1}$\\
\hline
2 & $U_{2,2}$\\
\hline
3 & $U_{2,3}$, $U_{3,3}$\\
\hline
4 & $U_{2,4}$, $U_{3,4}$, $U_{2,3}\oplus U_{1,1}$\\
\hline
5 & $U_{2,4}\oplus U_{1,1}$, $P(U_{2,3},U_{2,3})$, $U_{2,4}\oplus_2 U_{2,3}$\\
\hline
6 & $P(U_{2,4},U_{2,3})$, $M(K_4)$, $\mathcal{W}^3$\\
\hline
7 & $O_7$\\
\hline
\end{tabular}
\caption{Choices for $K$, the complement of $\si(N)$ in $PG(2,3)$.}
\label{o7restrictions}
\end{figure}

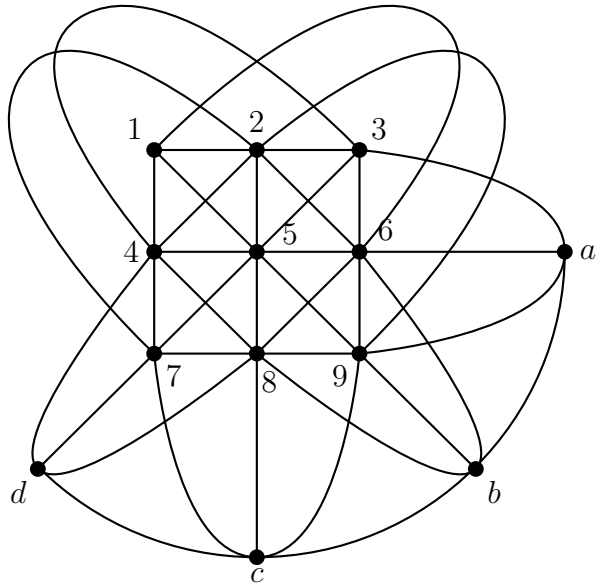
\begin{figure}
\centering
\begin{tikzpicture}[style=thick,scale=0.9]

\foreach \x in {0,1.5,3}
\foreach \y in {0,1.5,3} \draw[fill=black] (\x,\y) circle (0.1);

\foreach \x in {0,1.5,3} \draw (\x,0)--(\x,3);
\foreach \y in {0,1.5,3} \draw (0,\y)--(3,\y);

\draw plot [smooth, tension=1.5] coordinates {(3,3) (6,1.5) (3,0)};
\draw (3,1.5)--(6,1.5);
\draw[fill=black] (6,1.5) circle (0.1);

\draw plot [smooth, tension=1.5] coordinates {(0,0) (1.5,-3) (3,0)};
\draw (1.5,0)--(1.5,-3);
\draw[fill=black] (1.5,-3) circle (0.1);

\draw (1.5,3)--(3,1.5) (0,1.5)--(1.5,0);
\draw plot [smooth, tension=1.25] coordinates {(3,1.5) (4.7,-1.7) (1.5,0)};
\draw (0,3)--(4.7,-1.7);
\draw[fill=black] (4.7,-1.7) circle (0.1);

\draw (1.5,3)--(0,1.5) (3,1.5)--(1.5,0);
\draw plot [smooth, tension=1.25] coordinates {(0,1.5) (-1.7,-1.7) (1.5,0)};
\draw (3,3)--(-1.7,-1.7);
\draw[fill=black] (-1.7,-1.7) circle (0.1);

\draw plot [smooth, tension=1.75] coordinates {(0,0) (-2,4) (1.5,3)};
\draw plot [smooth, tension=1.75] coordinates {(0,1.5) (-1,5) (3,3)};
\draw plot [smooth, tension=1.75] coordinates {(3,0) (5,4) (1.5,3)};
\draw plot [smooth, tension=1.75] coordinates {(3,1.5) (4,5) (0,3)};

\draw (6,1.5) arc (0:-135:4.5);

\draw (0,3) node [above left] {1};
\draw (1.5,3.1) node [above] {2};
\draw (3,3) node [above right] {3};
\draw (-0.05,1.5) node [left] {4};
\draw (1.7,1.45) node [above right] {5};
\draw (3.1,1.5) node [above right] {6};
\draw (0,0) node [below right] {7};
\draw (1.4,-0.1) node [below right] {8};
\draw (3,0) node [below left] {9};
\draw (6.05,1.5) node [right] {$a$};
\draw (4.7,-1.7) node [below right] {$b$};
\draw (1.5,-3) node [below] {$c$};
\draw (-1.7,-1.7) node [below left] {$d$};

\end{tikzpicture}
\caption{The matroid $PG(2,3)$.}
\label{PG23}
\end{figure}

If $K=U_{0,0}$, then $\si(N)=PG(2,3)$. Let $\si(N)$ be labelled as in Figure~\ref{PG23}. Suppose $N^\ast$ has a consistent ordering, and let $B_x$, $B_y$, and $B_z$ be $S_x$-, $S_y$-, and $S_z$-blocks in a common circuit $C$ of $N^\ast$, where $x$, $y$, and $z$ are elements of $E(\si(N))$. Assume also that $B_y$ is adjacent to $B_x$ and $B_z$. Then, by Lemma~\ref{seriesefglemma}, $\co(N^\ast)$ does not have the series $(x,y,z)$-property. We show next that

\begin{sublemma}
\label{xyzPG23}
$x$, $y$, and $z$ are collinear in $\si(N)$, and $x\neq z$.
\end{sublemma}

Suppose $x$, $y$, and $z$ are not collinear in $\si(N)$. Then one easily finds circuits of $\co(N^\ast)$ that verify the series $(x,y,z)$-property in $\co(N^\ast)$, a contradiction. Similarly, when $x=z$ there are circuits of $\co(N^\ast)$ that verify the series $(x,y,z)$-property in $\co(N^\ast)$, a contradiction. Thus,~\ref{xyzPG23} holds.

By symmetry, we may assume that $C$ is the circuit $\{1,2,3,4,5,6,7,8,9\}$ of $\co(N^\ast)$; let $C'$ be the corresponding circuit of $N^\ast$. Consider an $S_1$-block $B$ in $C'$. The block $B$ is adjacent to an $S_e$- and $S_f$-block for some $e$ and $f$ in $C-1$. By~\ref{xyzPG23}, the elements $1$, $e$, and $f$ are collinear in $\si(N)$; without loss of generality, say $e=2$ and $f=3$. Let $B_3$ be the $S_3$-block adjacent to $B$. By repeatedly applying~\ref{xyzPG23}, we have that $B_3$ is adjacent to an $S_2$-block $B_2$, the block $B_2$ is adjacent to another $S_1$-block $B_1$, the block $B_1$ is adjacent to another $S_3$-block, and so on. It follows that $C'$ has a proper subset $X$ of elements not adjacent to any element of $C'-X$, a contradiction.

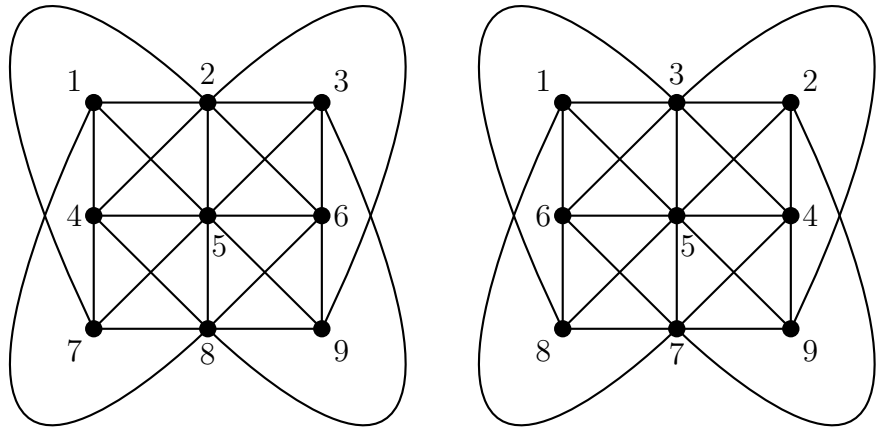
\begin{figure}[ht]
\centering
\begin{tikzpicture}[style=thick]

\foreach \x in {0,1.5,3}
\foreach \y in {0,1.5,3}\draw[fill=black] (\x,\y) circle (0.1);

\foreach \x in {0,1.5,3}\draw (\x,0)--(\x,3);
\foreach \y in {0,1.5,3}\draw (0,\y)--(3,\y);

\draw (0,0)--(3,3) (0,3)--(3,0);
\draw (1.5,0)--(3,1.5)--(1.5,3)--(0,1.5)--(1.5,0);

\draw plot [smooth, tension=1.2] coordinates {(0,0) (-1,4) (1.5,3)};
\draw plot [smooth, tension=1.2] coordinates {(3,0) (4,4) (1.5,3)};
\draw plot [smooth, tension=1.2] coordinates {(0,3) (-1,-1) (1.5,0)};
\draw plot [smooth, tension=1.2] coordinates {(3,3) (4,-1) (1.5,0)};

\draw (0,3) node [above left] {1};
\draw (1.5,3.1) node [above] {2};
\draw (3,3) node [above right] {3};
\draw (0,1.5) node [left] {4};
\draw (1.65,1.1) node {5};
\draw (3,1.5) node [right] {6};
\draw (0,0) node [below left] {7};
\draw (1.5,-0.05) node [below] {8};
\draw (3,0) node [below right] {9};

\end{tikzpicture}
\begin{tikzpicture}[style=thick]

\foreach \x in {0,1.5,3}
\foreach \y in {0,1.5,3}\draw[fill=black] (\x,\y) circle (0.1);

\foreach \x in {0,1.5,3}\draw (\x,0)--(\x,3);
\foreach \y in {0,1.5,3}\draw (0,\y)--(3,\y);

\draw (0,0)--(3,3) (0,3)--(3,0);
\draw (1.5,0)--(3,1.5)--(1.5,3)--(0,1.5)--(1.5,0);

\draw plot [smooth, tension=1.2] coordinates {(0,0) (-1,4) (1.5,3)};
\draw plot [smooth, tension=1.2] coordinates {(3,0) (4,4) (1.5,3)};
\draw plot [smooth, tension=1.2] coordinates {(0,3) (-1,-1) (1.5,0)};
\draw plot [smooth, tension=1.2] coordinates {(3,3) (4,-1) (1.5,0)};

\draw (0,3) node [above left] {1};
\draw (1.5,3.1) node [above] {3};
\draw (3,3) node [above right] {2};
\draw (0,1.5) node [left] {6};
\draw (1.65,1.1) node {5};
\draw (3,1.5) node [right] {4};
\draw (0,0) node [below left] {8};
\draw (1.5,-0.05) node [below] {7};
\draw (3,0) node [below right] {9};

\end{tikzpicture}
\vspace{-0.5cm}
\caption{Two geometric representations of the matroid $AG(2,3)$.}
\label{AG23}
\end{figure}

If $K=U_{2,4}$, then $\si(N)=AG(2,3)$. Figure~\ref{AG23} gives two labelled copies of $\si(N)$ in order to illustrate some of the symmetries of this matroid. Using Corollary~\ref{seriesefgcor}, let $c$ be the element 1 in the circuit $C=\{1,2,3,4,5,6\}$ of $\co(N^\ast)$. When $e=g=2$, the circuits $C$ and $\{1,2,3,7,8,9\}$ certify that $\co(N^\ast)$ has the series $(2,1,2)$-property with $D=\{1,3,4,6,7,9\}$. Since $\co(N^\ast)$ has a doubly transitive automorphism group, it follows that $\co(N^\ast)$ has the series $(e,1,e)$-property for each $e$ in $C$. When $\{e,g\}=\{2,3\}$, the circuits $C$ and $\{1,2,3,7,8,9\}$ certify that $\co(N^\ast)$ has the series $(2,1,3)$-property with $D=\{1,3,4,6,7,9\}$. The circuits $C$, $\{1,2,4,5,7,8\}$, and $\{1,2,4,6,8,9\}$ certify that $\co(N^\ast)$ has the series $(2,1,4)$-property with $D=\{1,3,4,6,7,9\}$. A symmetric set of circuits certifies that $\co(N^\ast)$ has the series $(e,1,g)$-property for each independent set $\{e,1,g\}$ contained in $C$. By Corollary~\ref{seriesefgcor}, $N^\ast$ is not orderable.

If $K=U_{2,4} \oplus U_{1,1}$, then $\si(N)=AG(2,3)\del 9$ with $AG(2,3)$ labelled as in Figure~\ref{AG23}. Using Corollary~\ref{seriesefgcor} again, let $c$ be the element 1 in the circuit $C=\{1,2,3,7,8\}$ of $\co(N^\ast)$. When $e=g=2$, the circuits $C$ and $\{1,2,4,6,8\}$ certify that $\co(N^\ast)$ has the series $(2,1,2)$-property with $D=\{1,3,4,6,7\}$. A symmetric set of circuits certifies that $\co(N^\ast)$ has the series $(e,1,e)$-property for each $e$ in $C-1$.

When $\{e,g\}=\{2,3\}$, the circuits $C$ and $\{1,2,4,6,8\}$ certify that $\co(N^\ast)$ has the series $(2,1,3)$-property with $D=\{1,3,4,6,7\}$. From Figure~\ref{AG23}, we see that the cases $\{e,g\}=\{2,7\}$ and $\{e,g\}=\{3,8\}$ are symmetric; and the circuits $C$ and $\{1,2,4,5,7,8\}$ certify that $\co(N^\ast)$ has the series $(2,1,7)$-property with $D=\{1,3,4,5,8\}$. The cases $\{e,g\}=\{2,8\}$ and $\{e,g\}=\{3,7\}$ are also symmetric, and the circuits $C$ and $\{1,2,4,6,8\}$ certify that $\co(N^\ast)$ has the series $(2,1,8)$-property with $D=\{1,3,4,6,7\}$. Finally, the circuits $\{1,2,4,5,7,8\}$ and $\{1,3,5,6,7,8\}$ certify that $\co(N^\ast)$ has the series $(7,1,8)$-property with $D=\{1,2,3,4,5,6\}$. Corollary~\ref{seriesefgcor} now implies $N^\ast$ is not orderable.

The next five cases make frequent use of Proposition~\ref{ordimpbalanced}(ii). The strategy is to contract strategic parallel classes of $N$ to get parallel extensions of $U_{2,4}$. These parallel extensions are dual to orderable series extensions of $U_{2,4}$, and Proposition~\ref{ordimpbalanced}(ii) implies that the parallel classes of such a parallel extension have the same size. For each case, we view $\si(N)$ as a restriction of the labelled copy of $PG(2,3)$ in Figure~\ref{PG23}. For each element $e$ in $E(N)$, let $p_e$ be the size of the parallel class of $N$ containing $e$.

If $K=U_{1,1}$, then $\si(N)=PG(2,3)\del d$. The following equations are obtained by applying Proposition~\ref{ordimpbalanced}(ii) in the minors $N/\cl(\{a\})$, $N/\cl(\{b\})$, and $N/\cl(\{c\})$, respectively:
\begin{equation*}
p_b+p_c=p_1+p_2+p_3=p_4+p_5+p_6=p_7+p_8+p_9;
\end{equation*}
\begin{equation*}
p_a+p_c=p_1+p_5+p_9=p_2+p_6+p_7=p_3+p_4+p_8;
\end{equation*}
\begin{equation*}
p_a+p_b=p_3+p_6+p_9=p_2+p_5+p_8=p_1+p_4+p_7.
\end{equation*}

Combining these equations, we obtain
\begin{equation*}
3(p_a+p_b)+3(p_a+p_c)+3(p_b+p_c)=3(p_1+p_2+\cdots+p_9),
\end{equation*}
which implies
\begin{equation*}
2(p_a+p_b+p_c)=p_1+p_2+\cdots+p_9,
\end{equation*}
and therefore
\begin{equation*}
3(p_a+p_b+p_c)=\vert E(N)\vert.
\end{equation*}

We conclude that exactly one-third of the elements of $E(N)$ lie on the line $\{a,b,c\}$. By symmetry, the same is true of the lines $\{1,6,8\}$, $\{3,5,7\}$, and $\{2,4,9\}$, so now four disjoint lines each account for one-third of the elements in $N$, a contradiction.

If $K=U_{2,3}$, then $\si(N)=PG(2,3)\del \{b,c,d\}$. The following equations are obtained by applying Proposition~\ref{ordimpbalanced}(ii) in the minors $N/\cl(\{1\})$, $N/\cl(\{2\})$, and $N/\cl(\{3\})$, respectively:
\begin{equation}
\label{case-u23-1}
p_2+p_3+p_a=p_4+p_7=p_6+p_8=p_5+p_9=\frac{1}{4}\big(\vert E(N)\vert-p_1\big);
\end{equation}
\begin{equation}
\label{case-u23-2}
p_1+p_3+p_a=p_4+p_9=p_5+p_8=p_6+p_7=\frac{1}{4}\big(\vert E(N)\vert-p_2\big);
\end{equation}
\begin{equation*}
p_1+p_2+p_a=p_5+p_7=p_6+p_9=p_4+p_8=\frac{1}{4}(\vert E(N)\vert-p_3).
\end{equation*}
Solving equations~(\ref{case-u23-1}) and~(\ref{case-u23-2}) for $\vert E(N)\vert$, we see that
\begin{equation*}
p_1+4p_2+4p_3=4p_1+p_2+4p_3,
\end{equation*}
so $p_1=p_2$. Through additional substitutions, it follows that $p_i=p_j$ for each $i,j\neq a$. But now $p_a=0$, a contradiction.

If $K=P(U_{2,3},U_{2,4})$, then $\si(N)=PG(2,3)\del \{7,8,9,a,b,d\}\cong P_7$. From the minors $N/\cl(\{1\})$ and $N/\cl(\{3\})$ and Proposition~\ref{ordimpbalanced}(ii), we get the equations
\begin{align*}
p_2+p_3&=p_4+p_c=p_5=p_6,\\
p_1+p_2&=p_6+p_c=p_4=p_5.
\end{align*}
It follows that $p_c=0$, a contradiction.

If $K=\mathcal{W}^3$, then $\si(N)=PG(2,3)\del \{6,8,9,b,c,d\}\cong O_7$. From the minors $N/\cl(\{7\})$ and $N/\cl(\{5\})$ we get the equations
\begin{equation*}
p_1+p_4=p_3+p_5=p_2=p_a,
\end{equation*}
\begin{equation*}
p_3+p_7=p_4+p_a=p_1=p_2.
\end{equation*}
It follows that $p_4=0$, a contradiction.

If $K=O_7$, then $\si(N)=PG(2,3)\del 6,8,9,a,b,c,d\cong\mathcal{W}^3$. From the minors $N/\cl(\{2\})$ and $N/\cl(\{4\})$ we get the equations
\begin{equation*}
p_1+p_3=p_4=p_5=p_7,
\end{equation*}
\begin{equation*}
p_1+p_7=p_2=p_3=p_5,
\end{equation*}
so $p_1=0$, a contradiction.

For the next six cases, we continue to view $N$ as a parallel extension of a restriction of $PG(2,3)$, with $PG(2,3)$ labelled as in Figure~\ref{PG23}. However, we now represent the deletion of an element $e$ from $PG(2,3)$ by setting $p_e$ to be 0. Each of these cases is eliminated using the following assertion.

\begin{sublemma}
\label{thefallout}
Let $N$ be a restriction of $PG(2,3)$ such that
\begin{itemize}
\item[(i)]{$p_c=p_d=0$;}
\item[(ii)]{$p_x\neq 0$ for each $x$ in $\{a,b,1\}$;}
\item[(iii)]{$\si(N/\cl(\{x\}))\cong U_{2,4}$ for each $x$ in $\{a,b,1\}$; and}
\item[(iv)]{$p_2$ and $p_3$ are not both zero.}
\end{itemize}
Then $N^\ast$ is not orderable.
\end{sublemma}

To see this, we first use the minors $N/\cl(\{a\})$ and $N/\cl(\{b\})$ to establish the equations
\begin{equation*}
p_b=p_1+p_2+p_3=p_4+p_5+p_6=p_7+p_8+p_9,
\end{equation*}
\begin{equation*}
p_a=p_2+p_6+p_7=p_1+p_5+p_9=p_3+p_4+p_8,
\end{equation*}
from which we obtain
\begin{equation*}
3p_a=p_1+p_2+\cdots+p_9=3p_b,
\end{equation*}
so $p_a=p_b$, and $\vert E(N)\vert=5p_a$. Now, $N/\cl(\{1\})$ gives that
\begin{equation*}
\vert E(N)\vert-p_1=4(p_2+p_3+p_a),
\end{equation*}
and substituting $5p_a$ for $\vert E(N)\vert$ produces
\begin{equation*}
p_a=p_1+4(p_2+p_3).
\end{equation*}
Finally, since $p_a=p_1+p_2+p_3$, we deduce that $p_2+p_3=0$, a contradiction. Thus~\ref{thefallout} holds.

The six options for $K$ eliminated by~\ref{thefallout} are the matroids $U_{2,2}$, $U_{3,3}$, $U_{3,4}$, $U_{2,3}\oplus U_{1,1}$, $P(U_{2,3},U_{2,3})$, and $U_{2,3}\oplus_2 U_{2,4}$. It is straightforward to check that, for each $K$ in this list, we may set classes of $PG(2,3)$ equal to zero in such a way that the zeroed classes form a restriction isomorphic to $K$, and the conditions of~\ref{thefallout} hold. For example, $U_{2,3}\oplus_2 U_{2,4}$ is produced when $p_5$, $p_7$, $p_9$, $p_c$, and $p_d$ are the zeroed classes.

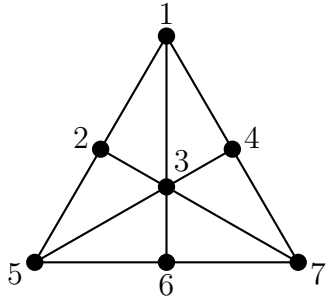
\begin{figure}
\centering

\begin{tikzpicture}[style=thick]

\foreach \x in {0,1,2} \draw[fill=black] (90+120*\x:2) circle (0.1);
\foreach \x in {0,1,2} \draw[fill=black] (-90+120*\x:1) circle (0.1);
\draw[fill=black] (0,0) circle (0.1);

\foreach \x in {0,1,2} \draw (90+120*\x:2)--(-90+120*\x:1) (90+120*\x:2)--(210+120*\x:2);

\draw (60:0.4) node {$3$};
\draw (90:2.3) node {$1$};
\draw (210:2.3) node {$5$};
\draw (330:2.3) node {$7$};
\draw (150:1.3) node {$2$};
\draw (-90:1.3) node {$6$};
\draw (30:1.3) node {$4$};

\end{tikzpicture}
\caption{The matroid $F_7^-$.}
\label{nonfano}
\end{figure}

In the final case, $K=M(K_4)$ and $\si(N)=F_7^-$. Label $F_7^-$ as in Figure~\ref{nonfano}, and, for each $e$ in $[7]$, let $S_e=S_e(N^\ast)$. Suppose $N^\ast$ has a consistent ordering, and let $B$ be an $S_1$-block in the ordering. In $N^\ast$, there is a circuit corresponding to each circuit $\{1,2,3,5,7\}$, $\{1,3,4,5,7\}$, and $\{1,3,5,6,7\}$ of $\co(N^\ast)$; let $\mathcal{X}$ be the collection of these circuits of $N^\ast$. Similarly, let $\mathcal{Y}$ be the collection of circuits of $N^\ast$ corresponding to the circuits $\{1,2,3,4\}$, $\{1,4,5,6\}$, and $\{1,2,6,7\}$ of $\co(N^\ast)$.

Suppose $B$ is adjacent to an $S_e$-block for some $e$ in $\{2,4,6\}$. Then the consistency of the circuits in $\mathcal{X}$ implies that $B$ is adjacent to an $S_e$-block for every $e$ in $\{2,4,6\}$. The circuits in $\mathcal{Y}$ now imply that $B$ is adjacent to an $S_2$-, $S_4$-, and $S_6$-block. Further, $B$ is not adjacent to an $S_e$-block for any $e$ in $\{3,5,7\}$. It follows that, in the circuit of $N^\ast$ corresponding to $\{1,2,3,5,7\}$, the block $B$ must be adjacent to a pair of $S_2$-blocks, contradicting the fact that $B$ is adjacent to both an $S_2$-block and an $S_4$-block in the circuit of $N^\ast$ corresponding to $\{1,2,3,4\}$.

We now know that, for each $e$ in $\{2,4,6\}$, the block $B$ is not adjacent to an $S_e$-block. The circuits in $\mathcal{Y}$ now imply that, in the circuit of $N^\ast$ corresponding to $\{1,2,3,5,7\}$, the block $B$ is adjacent to an $S_e$-block for every $e$ in $\{3,5,7\}$. This contradiction implies $N^\ast$ is not orderable.
\end{proof}

The next proposition is a result of Oxley~\cite{oxleynonbin} (see also ~\cite[Corollary 12.2.18]{oxley}). We will use it to prove Theorem~\ref{nonbin3conn}.

\begin{proposition}
\label{nonbinmanyminor}
A $3$-connected non-binary matroid whose rank and corank exceed two has a minor isomorphic to one of $\mathcal{W}^3$, $P_6$, $Q_6$, and $U_{3,6}$.
\end{proposition}

\begin{proof}[Proof of Theorem~\ref{nonbin3conn}]
Assume that the theorem fails for $M$. Then $r(M)\geq 3$. As $P_6$, $Q_6$, and $U_{3,6}$ each have $U_{3,5}$ as a minor, Proposition~\ref{nonbinmanyminor} and Propositions~\ref{noU35minor} and~\ref{no3whirlminor} now imply that $r^\ast(M)\leq 2$, so $r^\ast(M)=2$. As $M$ is 3-connected, it follows that $M\cong U_{n-2,n}$ for some $n\geq 5$. Hence $M$ has a $U_{3,5}$-minor, a contradiction.
\end{proof}

If $\{M_1,M_2,\dots,M_n\}$ is a set of a matroids, then a \emph{matroid-labelled tree} with vertex set $\{M_1,M_2,\dots,M_n\}$ is a tree $T$ such that
\begin{enumerate}
\item[(i)]{if $e$ is an edge of $T$ with endpoints $M_i$ and $M_j$, then $E(M_i)\cap E(M_j)=\{e\}$, and $\{e\}$ is not a separator of $M_i$ or $M_j$; and}
\item[(ii)]{$E(M_i)\cap E(M_j)$ is empty if $M_i$ and $M_j$ are non-adjacent.}
\end{enumerate}
The matroids $M_1, M_2,\dots,M_n$ are called the \emph{vertex labels} of $T$. Now suppose $e$ is an edge of $T$ with endpoints $M_1$ and $M_2$. We obtain a new matroid-labelled tree $T/e$ by contracting $e$ and relabelling the resulting vertex with $M_1\oplus_2 M_2$. As 2-sum is associative, $T/X$ is well defined for all subsets $X$ of $E(T)$.

Let $T$ be a matroid-labelled tree with $V(T)=\{M_1,M_2,\dots,M_n\}$ and\\
$E(T)=\{e_1,e_2,\dots,e_{n-1}\}$. Then $T$ is a \emph{tree decomposition} of a connected matroid $M$ if
\begin{enumerate}
\item[(i)]{$E(M)=(E(M_1)\cup E(M_2)\cup\cdots\cup E(M_n))-\{e_1,e_2,\dots,e_{n-1}\}$;}
\item[(ii)]{$\vert E(M_i)\vert\geq 3$ for all $i$ unless $\vert E(M)\vert<3$, in which case $n=1$ and $M=M_1$; and}
\item[(iii)]{$M$ labels the single vertex of $T/E(T)$.}
\end{enumerate}
In this case, the elements $\{e_1,e_2,\dots,e_{n-1}\}$ are the \emph{edge labels} of $T$. The next theorem of Cunningham and Edmonds~\cite{cunned} (see also~\cite[Theorem 8.3.10]{oxley}) tells us that $M$ has a \emph{canonical tree decomposition}, unique to within relabelling of the edges.

\begin{theorem}
Let $M$ be a $2$-connected matroid. Then $M$ has a tree decomposition $T$ in which every vertex label is $3$-connected, a circuit, or a cocircuit, and there are no two adjacent vertices that are both labelled by circuits or are both labelled by cocircuits. Moreover, $T$ is unique to within relabelling of its edges.
\end{theorem}

Let $T$ be a tree decomposition of a matroid $M$, and let $N$ and $p$ be a vertex label and edge label of $T$, respectively. For the remainder of this section, we define $M_{p,N}$ and $M_{p,N}'$ to be the matroids such that $M=M_{p,N} \oplus_2 M_{p,N}'$ with basepoint $p$, where $E(M_{p,N})$ contains the subset of $E(M)$ corresponding to the component of $T\del p$ containing $N$. Notice that if the vertex labels $M_1$ and $M_2$ lie in different components of $T\del p$, then $M_{p,M_1} = M_{p,M_2}'$.

In the next four lemmas, $M$ is assumed to be a connected, orderable, non-binary matroid whose canonical tree decomposition is $T$.

\begin{lemma}
\label{unisizer}
Suppose that $T$ has a vertex label $U$ that is isomorphic to $U_{2,n}$ for some $n\geq 4$. Then, for all $e,f\in E(U)$,
\begin{enumerate}
\item[(i)]{$e$ is an edge label of $T$, unless $M$ is a parallel extension of $U_{2,n}$;}
\item[(ii)]{all circuits of $M_{e,U}'$ containing $e$ have the same size; and}
\item[(iii)]{the circuits of $M_{e,U}'$ containing $e$ have the same size as the circuits of $M_{f,U}'$ containing $f$.}
\end{enumerate}
\end{lemma}

\begin{proof}
We may assume that $M$ is not a parallel extension of $U_{2,n}$ otherwise (i) holds. For each element $y$ of $E(U)$ that labels an edge of $T$, let $C_y$ be a circuit of $M_{y,U}'$ that contains $y$. As $M$ is not a parallel extension of $U_{2,n}$, we may assume that $\vert C_x\vert\geq 3$ for some element $x$. Let $M''$ be the matroid that is obtained from $U$ by attaching each $C_y$ via 2-sum. This matroid is a restriction of $M$ having $C_x-x$ as a non-trivial series class. Moreover, $M''$ is a series extension of $U_{2,n}$ and it is orderable. Thus, by Proposition~\ref{ordimpbalanced}(ii), $M''$ is a balanced series extension of $U_{2,n}$. Hence (i) holds. Furthermore, $\vert C_x\vert=\vert C_y\vert\geq 3$ for all $y$ in $E(U)-\{x\}$. Parts (ii) and (iii) now follow without difficulty.
\end{proof}

The next lemma generalizes Lemma~\ref{unisizer}(ii) to arbitrary edges of $T$.

\begin{lemma}
\label{circuitsizer}
Suppose that $T$ has a vertex label $U$ that is isomorphic to $U_{2,n}$ for some $n\geq 4$, and suppose $e$ is an edge label of $T$. Then the circuits of $M_{e,U}'$ that contain $e$ all have the same size.
\end{lemma}

\begin{proof}
Let $N$ be the endpoint of $e$ in the same component of $T\del e$ as $U$. If $U=N$, then the assertion holds by Lemma~\ref{unisizer}(ii), so assume otherwise. Let $f$ be the label of the edge incident with $U$ that lies on the path connecting $U$ to $N$ in $T$. Next, let $T'$ be the subtree of $T\del\{e,f\}$ containing $N$, and let $M'$ be the matroid with tree decomposition $T'$.

Fix a circuit $C$ of $M'$ that contains $e$ and $f$. Observe that, for each circuit $D$ of $M_{e,N}'$ that contains $e$, there is a circuit $(D-e)\cup(C-e)$ of $M_{f,U}'$ that contains $f$. By Lemma~\ref{unisizer}(ii), the quantity $\vert (D-e)\cup(C-e)\vert$ is the same for each choice of $D$, so every such circuit $D$ has the same size.
\end{proof}

\begin{lemma}
\label{U2nhighlander}
The tree $T$ has exactly one $3$-connected non-binary vertex label, and this label is isomorphic to $U_{2,n}$ for some $n\geq 4$.
\end{lemma}

\begin{proof}
As $M$ is non-binary, it has at least one 3-connected non-binary vertex label $N$. For each element $y$ of $E(N)$ that labels an edge of $T$, let $C_y$ be a circuit of $M'_{y,N}$ that contains $y$. Let $M''$ be the matroid that is obtained from $N$ by attaching each $C_y$ via 2-sum. Then $M''$ is a restriction of $M$. Thus $M''$ is an orderable series extension of $N$. By Propositions~\ref{noU35minor},~\ref{no3whirlminor}, and~\ref{nonbinmanyminor}, $N\cong U_{2,n}$ for some $n\geq 4$. Now suppose $T$ has a pair of 3-connected non-binary vertex labels $N_1\cong U_{2,n_1}$ and $N_2\cong U_{2,n_2}$ with $n_1,n_2\geq 4$. Let $e_1$ and $e_2$ be the edge labels of $T$ incident with $N_1$ and $N_2$ that lie on the path connecting $N_1$ and $N_2$ in $T$.

By Lemma~\ref{unisizer}(ii), the circuits of $M_{e_1,N_1}'$ containing $e_1$ all have size $k$ and the circuits of $M_{e_2,N_2}'$ containing $e_2$ all have size $\ell$, where $k$ and $\ell$ are integers exceeding one. Let $\{e_1,x,y\}$ be a circuit of $N_1$. By Lemma~\ref{unisizer}(i), $x$ and $y$ are also edge labels of $T$; let $C_x$ be a circuit of $M_{x,N_1}'$ containing $x$, and $C_y$ be a circuit of $M_{y,N_1}'$ containing $y$. Then $k=\vert C_x\vert=\vert C_y\vert$ by Lemma~\ref{unisizer}(iii). Now there is a circuit of $M_{e_2,N_2}'$ containing $e_2$ that also contains $C_x-x$ and $C_y-y$. Thus, $\ell\geq 2(k-1)+1$. A symmetric argument gives that $k\geq 2(\ell-1)+1$, and substitution yields that $k\leq 1$, a contradiction.
\end{proof}

The next lemma rules out 3-connected binary vertex labels that are not circuits or cocircuits. It uses the following result of Seymour~\cite{pds}.

\begin{proposition}
\label{MK4rounded}
Let $M$ be a $3$-connected binary matroid with at least four elements. If $e\in E(M)$, then $M$ has an $M(K_4)$-minor using $e$.
\end{proposition}

\begin{lemma}
\label{nobinlabels}
No vertex of $T$ is labelled by a $3$-connected binary matroid with at least four elements.
\end{lemma}

\begin{proof}
Suppose $B$ is such a vertex label of $T$, let $U$ be the unique vertex label with $U\cong U_{2,n}$ and $n\geq 4$ given by Lemma~\ref{U2nhighlander}, and say $E(U)=\{e_1,e_2,\dots,e_n\}$. Let $p\in E(B)$ and $e_1\in E(U)$ be the labels of the edges incident with $B$ and $U$, respectively, that lie on the path connecting $B$ to $U$ in $T$. By Proposition~\ref{MK4rounded}, $B$ has a minor isomorphic to $M(K_4)$ that uses $p$.

This minor can be written in the form $B/I\del I^\ast$, where $I$ is independent in $B$ and $I^\ast$ is coindependent in $B$. This makes $B/I$ a rank-three binary matroid with $M(K_4)$ as a restriction, so after deleting the loops from $B/I$, we obtain a parallel extension of either $M(K_4)$ or $F_7$. Dually, after deleting the coloops from $B\del I^\ast$, we obtain a series extension of $M(K_4)$ or $F_7^\ast$. Thus $B$ has a restriction $N_1$ using $p$ that is a series extension of $M(K_4)$ or $F_7^\ast$.

Suppose $q$ is an edge label of $T$ that is used in $N_1$, and choose a circuit $C_q$ of $M_{q,B}'$ that contains $q$. Form the matroid $N_2$ from $N_1$ by replacing $q$ with $C_q-q$ in $E(N_1)$ for each $q$ in $E(N_1)-p$ that is an edge label of $T$. Then $N_2$ is a series extension of $M(K_4)$ or $F_7^\ast$ that appears as a restriction of $M_{p,B}$. Now, for each $i$ in $\{2,3\}$, let $C_{e_i}$ be a circuit of $M_{e_i,U}'$ that contains $e_i$. Then $M_{p,B}'$ has a circuit $C_p$ that contains $p$ and both $C_{e_2}-e_2$ and $C_{e_3}-e_3$. Form the matroid $N$ from $N_2$ by taking the 2-sum of $N_2$ and $C_p$ across the basepoint $p$. Then $N$ is a restriction of $M$ that is a series extension of $M(K_4)$ or $F_7^\ast$. For each element $x$ of $M(K_4)$ or $F_7^\ast$, let $S_x$ be $S_x(N)$. By Lemma~\ref{circuitsizer}, every circuit of $N_2$ that contains $p$ has the same size.

\begin{figure}
\centering
\centering
\begin{tikzpicture}[style=thick]

\foreach \x in {0,1,2} \draw (0,0)--(90+120*\x:1.8);
\foreach \x in {0,1,2} \draw (90+120*\x:1.8)--(210+120*\x:1.8);

\draw[fill=white] (0,0) circle (0.1);
\foreach \x in {0,1,2} \draw[fill=white] (90+120*\x:1.8) circle (0.1);

\draw (0,-1.1) node {$p$};
\draw (150:1.1) node {$a$};
\draw (100:0.9) node {$b$};
\draw (30:1.1) node {$c$};
\draw (197:0.9) node {$d$};
\draw (-20:0.9) node {$e$};

\end{tikzpicture}
\caption{$K_4$ in the proof of Lemma~\ref{nobinlabels}.}
\label{k4}
\end{figure}
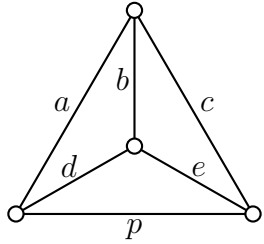

Suppose first that $N$ is a series extension of $M(K_4)$ with $K_4$ labelled as in Figure~\ref{k4}. Thus, every circuit of $N$ that contains $S_p$ has the same size. Since all circuits of $N$ containing $S_p$ have the same size, $\vert S_d\vert+ \vert S_e\vert=\vert S_a\vert + \vert S_b\vert + \vert S_e\vert$, so
\begin{equation}
\label{eq1}
\vert S_d\vert = \vert S_a\vert+ \vert S_b\vert.
\end{equation}
Similarly, $\vert S_a\vert + \vert S_c\vert = \vert S_d\vert + \vert S_b\vert + \vert S_c\vert$, so
\begin{equation}
\label{eq2}
\vert S_a\vert = \vert S_d\vert + \vert S_b\vert.
\end{equation}
Equations~(\ref{eq1}) and~(\ref{eq2}) imply that $\vert S_b\vert=0$, a contradiction.

\begin{figure}[b]
\centering
\begin{tikzpicture}[style=thick]

\draw[fill=black] (-3.5,2.5) circle (0.1);
\draw (-3.55,2.5) node[left] {$p$};
\draw[fill=black] (-1.75,3.75) circle (0.1);
\draw (-1.75,3.75) node[above left] {2};
\draw[fill=black] (1.75,3.75) circle (0.1);
\draw (1.75,3.75) node[above right] {3};
\draw[fill=black] (3.5,2.5) circle (0.1);
\draw (3.55,2.5) node[right] {4};
\draw[fill=black] (-1.75,2.25) circle (0.1);
\draw (-1.75,2.25) node [below left] {5};
\draw[fill=black] (-1.15,3.16) circle (0.1);
\draw (-1.2,3.55) node {6};
\draw[fill=black] (1.15,3.16) circle (0.1);
\draw (1.25,2.8) node {7};

\draw (-5,4.5)--(0,5.5)--(5,4.5)--(5,0.5)--(0,1.5)--(-5,0.5)--(-5,4.5);
\draw (0,1.5)--(0,5.5);
\draw (-3.5,2.5)--(0,5) (-3.5,2.5)--(0,2);
\draw (3.5,2.5)--(0,5);
\draw (-3.5,2.5)--(0,3.5)--(3.5,2.5);
\draw (1.75,3.75)--(0,2)--(-1.75,3.75);
\draw (-1.75,2.25)--(0,5);
\draw (-1.75,3.75).. controls (-2,1.75) ..(0,3.5);

\end{tikzpicture}
\caption{$F_7^\ast$ in the proof of Lemma~\ref{nobinlabels}.}
\label{nonbinf7star}
\end{figure}

Now suppose that $N$ is a series extension of $F_7^\ast$ with $F_7^\ast$ labelled as in Figure~\ref{nonbinf7star}. Since the circuits of $N$ containing $S_{p}$ must have the same size,
\begin{equation*}
\vert S_2\vert + \vert S_5\vert = \vert S_4\vert + \vert S_7\vert,
\end{equation*}
\begin{equation*}
\vert S_2\vert + \vert S_6\vert = \vert S_3\vert + \vert S_7\vert,
\end{equation*}
and
\begin{equation*}
\vert S_5\vert + \vert S_6\vert = \vert S_3\vert + \vert S_4\vert.
\end{equation*}
Together, these equations imply that
\begin{equation}
\label{s2equalss7}
\vert S_2\vert = \vert S_7\vert.
\end{equation}

Fix a consistent ordering of $M$. This induces a consistent ordering of $N$. Consider the circuit $C=S_p\cup S_2\cup S_3\cup S_4$ of $N$. Notice that $M$ has, as a restriction, a series extension $U'$ of $U_{2,n}$ whose ground set contains $C$. Specifically, $C=S_{e_1}\cup S_{e_2}\cup S_{e_3}$, where $S_{e_i}$ is $S_{e_i}(U')$.

Let $t$ be an arbitrary member of the series class $S_2$ of $N$. In $U'$, the element $t$ belongs to the class $S_{e_1}$, so $\{t\}$ is an $S_{e_1}$-block in the ordering of $C$ by Proposition~\ref{ordimpbalanced}(i). Lemma~\ref{uniformkeylemma}(i) implies that $t$ is adjacent to some element $x\in S_{e_2}$ and some $y\in S_{e_3}$; notice that, in $N$, the elements $x$ and $y$ both belong to $S_{p}$. Thus, every element of $S_2$ is adjacent to a pair of elements from $S_p$ in $C$. In particular, $t$ is not adjacent to any element of $S_2$ or of $S_3$. Now observe that $t$ is adjacent to this same pair $\{x,y\}$ in the circuit $S_p\cup S_2\cup S_5\cup S_6$ of $N$, so $t$ is also not adjacent to any element of $S_6$. It follows that $t$ is adjacent to a pair of elements from $S_7$ in the circuit $S_2\cup S_3\cup S_6\cup S_7$ of $N$. Therefore $\vert S_2\vert<\vert S_7\vert$, contradicting~(\ref{s2equalss7}).
\end{proof}

\begin{proposition}
\label{parpathord}
Let $M''$ be obtained from $M$ by parallel-path addition. Then $M$ is orderable if and only if $M''$ is orderable.
\end{proposition}
\begin{proof}
In forming $M''$ from $M$, let $P'$ be added in parallel to $P$. As $M''$ has $M$ as a restriction, $M$ is orderable if $M''$ is. Conversely, fix a consistent ordering of $M$ and let $C''$ be a circuit of $M''$. If $C''$ does not meet $P'$, give $C''$ the same ordering in $M''$ that it has in $M$. Otherwise, $C''$ contains $P'$ and either $C''=P\cup P'$, or there is a circuit $C$ of $M$ such that $C=(C''-P')\cup P$. In the the latter case, give $C''$ the same ordering in $M''$ that $C$ has in $M$ by replacing every element $p\in P$ by the corresponding element $p'\in P'$.

If $C''=P\cup P''$, take a circuit $D$ of $M$ containing $P$. Let $B_1$, $B_2$, \dots, $B_k$ be the $P$-blocks of $D$, numbered sequentially as they appear in a traversal of the ordering of $D$ in $M$. For each $i$ in $[k]$, let $B_i'=\{p':p\in B_i\}$. Now, order $C''$ as $B_1$, $B_1'$, $B_2$, $B_2'$, \dots, $B_k$, $B_k'$. It is straightforward to check that this gives a consistent ordering of $M''$.
\end{proof}

\begin{lemma}
\label{concurrencyexchange}
Let $M$ and $N$ be matroids, and let $S$ be a sequence of balanced series extensions and parallel-path additions by which $N$ is obtained from $M$. Suppose that the operation $s_1$ immediately precedes the operation $s_2$ in $S$. Then
\begin{itemize}
\item[(i)] if $s_1$ and $s_2$ are balanced series extensions of orders $m_1$ and $m_2$, then $s_1$ and $s_2$ may be replaced by a single balanced series extension of order $m_1m_2$; and
\item[(ii)] if $s_1$ is a parallel-path addition of size $k$, and $s_2$ is a balanced series extension of order $m$, then, in $S$, the order of the operations $s_1$ and $s_2$ can be reversed provided $s_1$ is replaced by a corresponding parallel-path addition of size $km$.
\end{itemize}
\end{lemma}

\begin{proof}
Part (i) is immediate. For part (ii), let $P_1$ be the $k$-element set that is added in parallel to the subset $P_2$ of a series class at step $s_1$. After the balanced series extension in step $s_2$ is performed, $P_1$ and $P_2$ become parallel paths $P_1'$ and $P_2'$ of size $mk$. Thus, the same result is obtained by first performing a balanced series extension of order $m$, then adding the $mk$-element set $P_1'$ in parallel to the subset $P_2'$ of a series class.
\end{proof}

We are now ready to prove the main result of the paper, which was given as Theorem~\ref{nonbinchar} in the introduction and is restated here for convenience.

\begin{theorem}
\label{nonbincharsec}
Let $M$ be a connected non-binary matroid. Then $M$ is orderable if and only if it can be obtained from $U_{2,n}$ for some $n\geq 4$ by a sequence of the following operations:
\begin{enumerate}
\item[(i)]{balanced series extension; and}
\item[(ii)]{parallel-path addition.}
\end{enumerate}
\end{theorem}

\begin{proof}
Let $n$ be an integer exceeding three, and let $M$ be a matroid obtained from $U_{2,n}$ by a sequence of balanced series extensions and parallel-path additions. Lemma~\ref{concurrencyexchange} implies that $M$ may be obtained from a balanced series extension of $U_{2,n}$ by a sequence of parallel-path additions, so, by Lemmas~\ref{balancebringsorder} and~\ref{parpathord}, $M$ is orderable. 

For the converse, we may assume that $M$ is simple, as adding an element in parallel is a parallel-path addition of size one. If $M\cong U_{2,n}$, the result holds, so assume otherwise. Let $T$ be the canonical tree decomposition of $M$. Lemmas~\ref{U2nhighlander} and~\ref{nobinlabels} imply that there is a single vertex label $U$ of $T$ for which $U\cong U_{2,n}$ and $n\geq 4$, and every vertex of $T-U$ is labelled by a circuit or a cocircuit. By Lemma~\ref{unisizer}(i), each $e$ in $E(U)$ labels an edge of $T$. Let $T_e'$ be the component of $T\del e$ that does not have $U$ as a vertex. As $M$ is simple, the leaves of $T$ are labelled by circuits. Therefore, if every $T_e'$ has only one vertex, then $M$ is a series extension of $U_{2,n}$, and the result holds by Proposition~\ref{ordimpbalanced}(ii). We show that, if this is not the case, then each $T_e'$ can be reduced to a single vertex labelled by a circuit via a sequence of deletions that can be undone by parallel-path additions.

Suppose $T_e'$ has at least two vertices. Since only one vertex of $T_e'$ is adjacent to $U$, not all vertices of $T_e'$ are leaves of $T$. We now observe that

\begin{sublemma}
\label{specialvert}
$T_e'$ has a vertex $v$ that
\begin{enumerate}
\item[(i)]{is adjacent to a leaf of $T$; and}
\item[(ii)]{has exactly one neighbor that is not a leaf of $T$}
\end{enumerate}
\end{sublemma}

If $L$ is the set of leaves of $T$, such a vertex $v$ can be found as a leaf of $T-L$. Since the leaves of $T$ are labelled by circuits and $T$ is canonical, $v$ is labelled by a cocircuit $C^\ast$. Lemma~\ref{circuitsizer} now implies that the circuits that label the leaves of $T$ adjacent to $C^\ast$ all have the same size, and every element of $C^\ast$ must be used as a basepoint labelling an edge of $T$. 

We can delete all but one of the leaves, $C$ say, of $T$ that are adjacent to $C^\ast$, along with the corresponding basepoints in $C^\ast$, since the circuit that labels each deleted leaf can be added via a parallel-path addition. As $C^\ast$ is now a pair of parallel elements, we can delete the leaf labelled $C$ and relabel $v$ with $C$. At this point, $v$ is a leaf, and is either adjacent to $U$, in which case the work on this subtree is complete, or $v$ is adjacent to another vertex of $T_e'$ labelled by a circuit $C'$. In the latter case, keep $T$ canonical by contracting the edge of $T$ between $v$ and $C'$ and labelling the resulting vertex with the circuit that is the 2-sum of $C$ and $C'$.

Provided the modification of $T_e'$ continues to have at least two vertices, condition~\ref{specialvert} continues to hold, and the process described in the previous paragraph can be repeated. Thus, we may assume $T_e'$ consists of a single vertex labelled by a circuit. By applying this pruning process on the other subtrees attached to $U$, the tree $T$ is reduced to the decomposition tree of a balanced series extension of $U_{2,n}$. Thus, $M$ can be obtained from a balanced series extension of $U_{2,n}$ by a sequence of parallel-path additions.
\end{proof}

\section{Theta-Orderability}
\label{theta}

Recall that theta-orderability of a matroid requires a consistent ordering of the matroid with respect to the theta-graphs of that matroid. Each of the elementary properties of orderability given in Proposition~\ref{basics} also holds for theta-orderability. Their straightforward proofs are omitted.

\begin{proposition}
\label{tbasics}
Let $M$ be a matroid.
\begin{enumerate}
\item[(i)]{If $M$ is theta-orderable, then $M\del e$ is theta-orderable for all $e$ in $E(M)$.}
\item[(ii)]{If $r(M)\leq 2$, then $M$ is theta-orderable.}
\item[(iii)]{$M$ is theta-orderable if and only if the connected components of $M$ are\\theta-orderable.}
\item[(iv)]{$M$ is theta-orderable if and only if $\si(M)$ is theta-orderable.}
\end{enumerate}
\end{proposition}

Next we prove Theorem~\ref{wag}, a characterization of graphic theta-orderable matroids.

\begin{proof}[Proof of Theorem~\ref{wag}]
It is clear that a graphic matroid is theta-orderable. Moreover, Wagner~\cite{wagnerarcs} proved that a matroid is graphic if and only if it has no set of incompatible arcs. Now suppose that $M$ has a circuit $C$ and a set $\{A_1,A_2,A_3\}$ of incompatible arcs of $C$. It remains to show that $M$ is not theta-orderable. Our proof of this is a straightforward modification of Wagner's proof that no graphic matroid has a set of incompatible arcs~\cite[Lemma 2]{wagnerarcs}. Assume that $M$ is theta-orderable. Because each of $A_1$, $A_2$, and $A_3$ is an arc, for each $i$ in $\{1,2,3\}$, there is a theta-graph of $M$ in which $A_i$ is a theta-arc. As $M$ is theta-orderable, $A_i$ is a block in a consistent ordering of $M$. As $\{A_1,A_2,A_3\}$ is an incompatible set, there are distinct elements $e_1$, $e_2$, and $e_3$ of $C$ such that $e\in A_1\cap A_2\cap A_3$ and $e_i\in A_i-(A_j\cup A_k)$ for all $\{i,j,k\}=\{1,2,3\}$. For each $h$ in $\{2,3\}$, the set $A_1\cup A_h$ is a block in $C$ in which $e$ appears between $e_1$ and $e_h$. Then $e$ does not appear between $e_2$ and $e_3$ in $A_2\cup A_3$, a contradiction.
\end{proof}

To prove Theorem~\ref{tononbinchar}, we will establish the following equivalent version of it.

\begin{theorem}
\label{tononbincharsec}
A simple connected non-binary matroid is theta-orderable if and only if it is a balanced series extension of $U_{2,n}$ for some $n\geq 4$.
\end{theorem}

\begin{proof}
First, for $n\geq 4$, the matroid $U_{2,n}$ and its series extensions have no theta-graphs. Therefore, consistent orderings of these matroids are also theta-orderings.

Conversely, suppose $M$ is a simple connected non-binary orderable matroid. By Theorem~\ref{nonbincharsec} and Lemma~\ref{concurrencyexchange}, for some $n\geq 4$, we can obtain $M$ from a balanced series extension $B$ of $U_{2,n}$ by a sequence of parallel-path additions. It now suffices to show that the sequence of parallel-path additions is empty.

Suppose to the contrary that $P'$ is a set added in parallel to a subset $P$ of a series class $S$ of $B$. Note $\vert P\vert\geq 2$ since $M$ is simple. Now, by Proposition~\ref{ordimpbalanced}, each $S$-block in a consistent ordering of $B$ contains a single element. As $B$ is a restriction of $M$, this implies that the elements of $P$ are not a block in a consistent ordering of $M$. Since $M$ has a theta-graph with $P$ and $P'$ as theta-arcs, this is a contradiction.
\end{proof}

\section{Characterizing 3-Connected Orderable Binary Matroids}
\label{progress}

This section proves the following partial result towards Conjecture~\ref{3connbinconj}. Theorem~\ref{4connreg} is an immediate consequence of this result.

\begin{theorem}
\label{4connbinF7}
A $4$-connected binary orderable matroid with no series minor isomorphic to $F_7^\ast$ is graphic.
\end{theorem}

Our proof will require the next three results, the first of which is due to Seymour~\cite{pdsadjacency}. Two elements are \emph{opposite} in $M(K_4)$ if they form a matching in the $K_4$.

\begin{theorem}
\label{adjtheorem}
Let $M$ be a $4$-connected binary matroid and let $e$ and $f$ be elements of $M$. Suppose there is no $M(K_4)$-minor of $M$ in which $e$ and $f$ are opposite elements. Then there is a graph $G$ with $M=M(G)$ or $M^\ast(G)$, and $e$ and $f$ are adjacent edges in $G$.
\end{theorem}

\begin{proposition}
\label{k4opp}
In a consistent ordering of a series extension $M$ of $M(K_4)$, if two elements correspond to opposite elements in the $M(K_4)$, then they are not adjacent.
\end{proposition}

\begin{figure}
\centering
\begin{tikzpicture}[style=thick]

\foreach \x in {0,1,2} \draw (0,0)--(90+120*\x:1.8);
\foreach \x in {0,1,2} \draw (90+120*\x:1.8)--(210+120*\x:1.8);

\draw[fill=white] (0,0) circle (0.1);
\foreach \x in {0,1,2} \draw[fill=white] (90+120*\x:1.8) circle (0.1);

\draw (0,-1.2) node {$Y$};
\draw (150:1.2) node {$A$};
\draw (105:0.8) node {$X$};
\draw (30:1.2) node {$D$};
\draw (197:0.9) node {$B$};
\draw (-17:0.9) node {$C$};

\end{tikzpicture}
\caption{$K_4$ in the proof of Proposition~\ref{k4opp}.}
\label{k4oppfig}
\end{figure}
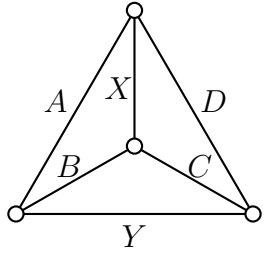

\begin{proof}
Let $A$, $B$, $C$, $D$, $X$, and $Y$ be the series classes of $M$, labelled as in Figure~\ref{k4oppfig}. Take elements $x$ in $X$ and $y$ in $Y$, and suppose $x$ and $y$ are adjacent in the given consistent ordering of $M$.

In the circuit $A\cup X\cup C\cup Y$, we have that $y$ is adjacent to at most one member of $C$. Therefore, in $B\cup C\cup Y$, there must be an element, $b_y$, of $B$ that is adjacent to $y$. Similarly, in $A\cup X\cup B$, there must be an element, $b_x$, of $B$ adjacent to $x$. Now, in $B\cup X \cup D\cup Y$, we have the block $b_xxyb_y$, so no member of $D$ is adjacent to $y$. By symmetry, no member of $A$ is adjacent to $y$. Since $y$ is adjacent to at most one element in $Y$, it follows that there is no second element of $A\cup D\cup Y$ adjacent to $y$, a contradiction.
\end{proof}

\begin{lemma}
\label{oppimpadj}
Suppose $M$ is a binary matroid with no series minor isomorphic to $F_7^\ast$. If $e$ and $f$ are opposite elements in an $M(K_4)$-minor of $M$, then $e$ and $f$ are not adjacent in any consistent ordering of $M$.
\end{lemma}

\begin{proof}
Assume that $M$ has a consistent ordering in which $e$ and $f$ are adjacent. Let $N$ be an $M(K_4)$-minor of $M$ in which $e$ and $f$ are opposite elements, and write $N=M\del X/Y$ with $X$ coindependent and $Y$ independent. Then $N^\ast=M^\ast/X\del Y$, where $r(M^\ast/X)=r(N^\ast)=3$. Since $M^\ast/X$ is binary and $N^\ast\cong M(K_4)$, if $L$ is the set of loops of $M^\ast/X$, then $M^\ast /X\del L$ is a parallel extension of either $M(K_4)$ or $F_7$. It follows that $M\del(X\cup L)$ is a series extension of $M(K_4)$ or $F_7^\ast$. By assumption, $M$ has no series minor isomorphic to $F_7^\ast$, so $M\del(X\cup L)$ is a series extension of $M(K_4)$. However, $e$ and $f$ are adjacent in the consistent ordering of $M\del(X\cup L)$ inherited from $M$ and correspond to opposite elements in $N$, a contradiction by Proposition~\ref{k4opp}.
\end{proof}

We now prove the main result of this section.

\begin{proof}[Proof of Theorem~\ref{4connbinF7}]
Let $M$ be a 4-connected binary orderable matroid that does not have $F_7^\ast$ as a series minor. Take a consistent ordering of $M$ and assume $M$ is not graphic. Suppose $M$ is cographic, letting $M=M^\ast(G)$ for some graph $G$. Take an edge $e$ of $G$ with endpoints $u$ and $v$. Let $(x_1\ x_2\ \cdots\ x_n\ e)$ be the ordering on the edges meeting $u$, and let $(e\ y_1\ y_2\ \cdots\ y_m)$ be the ordering on the edges meeting $v$. Then we may assume the ordering on the bond that is the symmetric difference of these two vertex bonds is $(x_1\ x_2\ \cdots\ x_n\ y_1\ y_2\ \cdots\ y_m)$, so $x_n$ and $y_1$ are adjacent. Combining Lemma~\ref{oppimpadj} and Theorem~\ref{adjtheorem}, we now have that $x_n$ and $y_1$ share an endpoint in $G$. Hence, $\{e,x_n,y_1\}$ is a triangle in $G$, a contradiction as $M$ is 4-connected.

We may now assume that $M$ is not cographic. Let $e$ and $f$ be adjacent elements of $M$. By Theorem~\ref{adjtheorem}, $e$ and $f$ appear as opposite elements in some $M(K_4)$-minor of $M$. Lemma~\ref{oppimpadj} now gives a contradiction.
\end{proof}

\subsection*{Acknowledgements}

The authors thank Jim Geelen for suggesting that~\cite{pdsadjacency} may be helpful in resolving Conjecture~\ref{3connbinconj}.


\end{document}